\documentclass[final,leqno,onefignum,onetabnum]{siamltex1213}
\usepackage{bbm}
\usepackage{amsmath,amsfonts,amssymb,enumerate}
\usepackage{color}
\usepackage{lscape}
\usepackage{latexsym}
\usepackage{multirow}
\usepackage{rotating}
\usepackage{threeparttable}
\usepackage{graphicx}
\usepackage{algpseudocode,algorithm}
\usepackage[margin=1in]{geometry}

\graphicspath{{figs/}}
\definecolor{doushalv}{rgb}{0.78,0.93,0.80}

\title{ Two-step Fixed-point Proximity Algorithms for Multi-block Separable Convex Problems\thanks{This research is supported in part by Guangdong Provincial Government of China through the ``Computational Science Innovative Research Team'' program, by the Natural Science Foundation of China under grants 11501584 and 11471013, and by the
Natural Science Foundation of Guangdong Province under grants 2014A030310332 and 2014A030310414.}}

\author {Qia Li 
\footnotemark[2]
\and Yuesheng Xu 
\footnotemark[2] \footnotemark[4]
\and
Na Zhang \footnotemark[3]
}

\begin{document}

\maketitle
\renewcommand{\thefootnote}{\fnsymbol{footnote}}
\footnotetext[3]{Department of Applied Mathematics, College of Mathematics and Informatics, South China Agricultural University, Guangzhou 510640, P. R. China. (\email{nzhsysu@gmail.com}). Questions, comments, or corrections to this document may be directed to that email address.}
\footnotetext[2]{Guangdong Province Key Laboratory of Computational Science, School of Data and  Computer Sciences, Sun Yat-sen University, Guangzhou 510275, P. R. China.}
\footnotetext[4]{This author is also a professor emeritus of mathematics in Syracuse University, Syracuse, NY 13244, USA.}

\begin{abstract}
\normalsize
Multi-block separable convex problems  recently received considerable attention. This class of optimization problems minimizes a separable convex objective function with linear constraints. The algorithmic challenges come from the fact that the classic alternating direction method of multipliers (ADMM) for the problem is not necessarily convergent. However, it is observed that  ADMM outperforms numerically many of its variants with guaranteed theoretical convergence. The goal of this paper is to develop convergent and computationally efficient algorithms for solving multi-block separable convex problems. We first characterize the solutions of the optimization problems by proximity operators of the convex functions involved in their objective function. We then design a two-step fixed-point iterative scheme for solving these problems based on the characterization. We further prove  convergence of the iterative scheme and show that it has $O(\frac{1}{k})$ convergence rate in the ergodic sense and the sense of the partial primal-dual gap, where $k$ denotes the iteration number. Moreover, we derive specific two-step fixed-point proximity algorithms (2SFPPA)  from the proposed iterative scheme and establish their global convergence.  Numerical experiments for solving the sparse MRI problem demonstrate the numerical efficiency of the proposed 2SFPPA.
\end{abstract}



\pagestyle{myheadings}
\thispagestyle{plain}

\section{Introduction}
We consider in this paper the convex minimization problem with linear constraints and a separable objective function in the form of the sum of several convex functions. For a positive  integer $d$, by $\mathbb{R}^d$ we denote the usual $d$-dimensional Euclidean space. The minimization problem we consider in this paper has the form \begin{equation}\label{model}
\min\left\{\sum _{i=1}^s f_i(x_{i}): \sum_{i=1}^s A_i x_i=b, x_i \in \mathbb{R}^{n_i}, i=1,2,\dots,s\right\},
\end{equation}
where $f_i: \mathbb{R}^{n_i}
\rightarrow \overline{\mathbb{R}}:=\mathbb{R}\cup\{+\infty\}$ is a
proper lower semicontinuous  convex function, $A_i$ is a given $m \times n_i$ real matrix, $n_i$ is the dimension of variable $x_i$, for $i=1,2,\dots,s$ and  $b\in\mathbb{R}^m$ is a given vector. Here, variable $x$ is decomposed into $s$ blocks, that is $x:=(x_1,x_2,\dots,x_s)$.

Many problems arising from image processing and machine learning can be cast into the form of model \eqref{model}. For example, the total-variation based image denoising model \cite{nikolova:SIAM-AM:00,rudin:icip:94}, sparse representation based image restoration \cite{Cai-Chan-Shen:ACHA:08,chan:sjsc:03,chan:acha:04,Li-Shen-Yang:ACHA:2012}, lasso regression \cite{Tibshirani:Lasso1996Regression} and support vector machines \cite{Vapnik:SVM} are special cases of problem \eqref{model} with $s=2$. In addition,  we also refer to \cite{Li-Zhang:ACHA:2015,Lustig-Donoho-Pauly:MRM:07,Ruszczy1993Parallel,Tibshirani2005Sparsity} for some applications of model \eqref{model} with $s\ge 3$.

The alternating direction method of multipliers (ADMM) \cite{Gabay1976Mercier} was originally proposed for solving  problem \eqref{model} with $s=2$, and was recently widely used in the area of image processing \cite{Cai-Osher-Shen:SAIMIS:09,Goldstein-Osher:SAIMIS:09,Sawatzky2014Proximal,Wen-Yin:ADMMSemiDef:MathProg2010}. Since ADMM requires inner iterations to solve its subproblems of ADMM,  its linearized version (LADMM) was proposed and was successfully  used in applications \cite{Esser-Zhang-Chan:1order-PD-09}.  As $s\ge 3$, one can directly extend the original ADMM (LADMM) to  problem \eqref{model}. Without an additional assumption, however, it was recently  shown in \cite{Chen-He-Ye-Yuan:MathProg2014} that the direct extension of ADMM to multi-block convex problems is not necessarily convergent, although it may work well in practice. Very recently, there were some investigations \cite{Davis-Yin:3operator,He-yuan:admmstrongly,Li-Sun-Toh:proxADMM2015,Lin-Ma-ZHang:admmrate} on  convergence of the extension of ADMM under some additional assumptions. Some researchers dedicated to modify  ADMM or LADMM to make it convergent. For instance, the Jacobian-type ADMM was proposed in \cite{deng-yin:parallelADMM} for parallel computing, the semi-proximal ADMM proposed in  \cite{Li-Sun-Toh:proxADMM2015,Sun-Toh-Yang:2014A} is for convex quadratic programming and conic programming, the Gaussian back substitution technique was proposed in \cite{He-Yuan:admmGBS:SIAMOPT,He-Yuan:LADMGBS:2013} to make ADMM and LADMM converge. It was shown in  \cite{He-Yuan:admmGBS:SIAMOPT,He-Yuan:LADMGBS:2013} the attractiveness of the Gaussian back substitution technique for theoretical analysis on  convergence of ADMM-type algorithms. However,  the numerical results show that the correction step is time consuming and the ADMM (LADMM) with Gaussian back substitution may require more iterations than the direct extension of ADMM (LADMM) to achieve the same  objective function value. Therefore, in this paper, we dedicate to establishing convergent and efficient algorithms.

As shown in  \cite{Attouch-Combettes:ForBack2010,Li-Shen-Xu-Zhang:AiCM:14,Li-Zhang:ACHA:2015,Micchelli-Shen-Xu:IP-11}, the notion of proximity operators provides  a useful tool for the algorithmic development due to its firmly nonexpansive property. ADMM was shown in \cite{Li-Shen-Xu-Zhang:AiCM:14} a special case of the  proximity algorithms.
Although the one-step fixed-point proximity algorithms proposed in \cite{Li-Shen-Xu-Zhang:AiCM:14} can be  applied to  model \eqref{model} directly, they do not utilize the separable property of the objective function, that is, the variable $x_1, x_2,\dots, x_s$ are updated simultaneously. In contrast,  ADMM takes advantage of the separability of the objective function and utilizes the block-wise Gauss-Seidel technique.
Thus, in order to develop convergent  algorithms for  problem \eqref{model},  we propose to develop  two-step fixed-point proximity algorithms. The term two-step means that when we update  values of the next step, we not only use  values of the current step  but also those of the previous step.
In one of our previous papers \cite{Li-Shen-Xu-Zhang:AiCM:14}, we designed a multi-step iterative scheme, introduced  the notions of weakly firmly nonexpansive operators and Condition-M (Semi-Condition-M), and presented the convergence results of the multi-step scheme with the help of the  notions. In this paper, we will follow the idea of \cite{Li-Shen-Xu-Zhang:AiCM:14} to develop convergent two-step fixed-point proximity algorithms.

This paper has the following contributions. First, we present  a characterization of the solutions of  problem \eqref{model} by fixed-points of a proximity related operator and  develop a two-step fixed-point iterative scheme based on the fixed-point equation.  Second, we prove  convergence of the proposed iterative scheme by the notions of weakly firmly nonexpansive and Condition-M proposed in \cite{Li-Shen-Xu-Zhang:AiCM:14}. We prove that as long as the matrices involved in the scheme satisfy Condition-M,  which can be easily verified, the iterative scheme  converges and the sequence $\{(x_1^k, \dots, x_s^k): k\in\mathbb{N}\}$ generated by the proposed algorithm  converges to a solution of  problem \eqref{model}.  Third, we analyze the convergence rate of the proposed iterative scheme. We prove that the scheme has $O(\frac{1}{k})$ ergodic convergence rate. In addition, the average of the sequence generated by the proposed scheme has $O(\frac{1}{k})$ convergence rate in the sense of  the primal-dual gap. Fourth, several specific convergent algorithms are designed from the iterative scheme, including the two-step implicit and explicit fixed-point proximity algorithms as well as their variants. Furthermore, we apply the proposed two-step fixed-point proximity algorithm  to the sparse MRI reconstruction problem. Numerical results show that the proposed two-step fixed-point proximity algorithm performs as efficiently as the direct extension of LADMM, which is not necessarily convergent.

We organize this paper in eight sections. In Section \ref{sec:characterization}, we characterize the solutions of  problem \eqref{model} by fixed-points of a proximity related operator. Based on this characterization, we develop in Section \ref{sec:algorithm} a two-step iterative scheme and prove its  convergence in Section \ref{sec:converge}. In Section \ref{sec:ConvRate}, we analyze the convergence rate of the proposed iterative scheme. We design several specific algorithms from the iterative scheme in Section \ref{sec:specificAlg} and apply in Section \ref{sec:exp} one of them to the sparse MRI reconstruction problem. We conclude this paper in Section \ref{sec:conclusion}.

\section{A Characterization of  Solutions of the Minimization Problem}\label{sec:characterization}

In this section we  present a characterization of
solutions of model \eqref{model} in terms of a system of fixed-point equations via the
proximity operators of the functions involved in the objective function. The system of fixed-point equations will serve
as a basis for developing iterative schemes for solving
the problem.

We now recall the notion of the proximity operator of a convex
function. For $x$ and $y$ in $\mathbb{R}^d$, we denote the standard inner product by $\langle x, y \rangle:=\sum_{i\in\mathbb{N}_d}x_iy_i$, where $\mathbb{N}_d:=\{1,2,\dots,d\}$ and the standard $\ell_2$-norm by $\|x\|_2:=\langle x,x\rangle^\frac{1}{2}$. By   $\mathbb{S}^d_{+}$ we denote the
set of symmetric positive   definite matrices.
For an $H\in\mathbb{S}^d_{+}$ the $H$-weighted inner product is defined
by $ \langle x,y \rangle_H := \langle x,Hy \rangle $  and the
corresponding $H$-weighted $\ell_2$-norm is defined by $\|x\|_{H}:={\langle x,x
\rangle_H}^{\frac12}$. For a $d\times \ell$ matrix $A$, we define $\|A\|_2$ as the largest singular value of $A$. By $\Gamma_0(\mathbb{R}^d)$ we denote the
class of all lower semicontinuous proper convex functions $\varphi: \mathbb{R}^d
\rightarrow \overline {\mathbb{R}}$. For a function $\varphi \in
\Gamma_0(\mathbb{R}^d)$, the proximity operator of $\varphi$ with respect
to a given matrix $H \in \mathbb{S}^d_{+}$, denoted by
$\mathrm{prox}_{\varphi,H}$, is a mapping from $\mathbb{R}^d$ to itself,
defined for a given point $x \in \mathbb{R}^d$ by
\begin{equation}\label{def:prox}
\mathrm{prox}_{\varphi,H} (x):=\mathop{\mathrm{argmin}} \left\{\frac{1}{2} \|u-x\|^2_H + \varphi(u): u \in \mathbb{R}^d \right\}.
\end{equation}
In particular, we use $\mathrm{prox}_{\varphi}$ for $\mathrm{prox}_{\varphi,I}$.

The proximity operator of a function is intimately related to its
subdifferential. The subdifferential of a function $\varphi$ at a given
vector $x \in \mathbb{R}^d$ is the set defined by
$$
\partial \varphi(x):=\{y:~y\in \mathbb{R}^d,\ \mbox{and} \ \varphi(z)\geq \varphi(x)+\langle y,z-x\rangle, \  \mbox{for all} \ z\in \mathbb{R}^d \}.
$$
We remark that if a function $\varphi \in \Gamma_0(\mathbb{R}^d)$ is Fr\'{e}chet
differentiable at a point $x\in\mathbb{R}^d$ then $\partial \varphi(x)=\{\nabla
\varphi(x)\}$, where $\nabla \varphi (x)$ is the Fr\'{e}chet gradient of $\varphi$.
It is shown that for any $H\in\mathbb{S}_+^d$, $x\in\mathrm{dom}(\varphi)$ and $y \in \mathbb{R}^d$,
\begin{equation}\label{eq:sub-prox-0}
Hy\in\partial {\varphi}(x) \quad \mbox{if and only if} \quad
x=\mathrm{prox}_{\varphi, H}(x+y).
\end{equation}
For a discussion of this relation, see, e.g.,  \cite[Proposition 16.34]{Bauschke-Combettes:11} or  \cite{Micchelli-Shen-Xu:IP-11}.

The proximity operator plays a crucial role in convex analysis and applications (see, e.g., \cite{moreau:RASPS:62,Rockafellar:SIAMCO:1976}).   Recall  that  operator $J$ is called firmly nonexpansive (resp., nonexpansive) with respect to a given  matrix $H\in\mathbb{S}_+^d$ if for all $x, y\in\mathbb{R}^d$
 $$
 \|Jx-Jy\|_H^2\le \langle Jx-Jy, x-y\rangle_H ~~(\mathrm{resp.},~~ \|Jx-Jy\|_H\le \|x-y\|_H).
 $$
 We remark here that the symmetric positive definite matrix  $H$ defines specific inner-product of the Hilbert space $\mathbb{R}^d$ and if $H=I$ we do not specify the matrix $H$ for simplicity.  As shown in \cite{Bauschke-Combettes:11}, the proximity operator of a convex function is firmly nonexpansive and is contractive when the function is strongly convex.

We also need the notion of the conjugate function. The conjugate of
$\varphi\in \Gamma_0(\mathbb{R}^d)$ is the function $\varphi^* \in
\Gamma_0(\mathbb{R}^d)$ defined at $y \in \mathbb{R}^d$ by
$
\varphi^*(y):= \sup\{\langle x, y \rangle-\varphi(x): x\in \mathbb{R}^d\}.
$
A characterization of the subdifferential of a function $\varphi$ in
$\Gamma_0(\mathbb{R}^d)$ is that for $x \in \mathrm{dom}(\varphi)$ and $y
\in \mathrm{dom}(\varphi^*)$
\begin{equation}\label{eq:dual-sub}
y \in \partial \varphi(x) \quad \mbox{if and only if} \quad x \in \partial
\varphi^*(y).
\end{equation}
The notion of  the indicator function is also required. For a set $S\subseteq\mathbb{R}^d$, the indicator function on $S$, at point $x$, is defined as
 $$
 \iota_S(x):=\begin{cases}
 0,&\mathrm{~if~} x\in S,\\
 +\infty,&\mathrm{~else}.
 \end{cases}
 $$
Moreover, we denote the smallest cone in $\mathbb{R}^d$ containing $S$ by $\mathrm{cone}(S)$. Then the relative interior of $S$ (see Definition 6.9 of \cite{Bauschke-Combettes:11}) is defined as
$$
\mathrm{ri}(S):=\{x\in S : \mathrm{cone}(S-x) = \mathrm{span}(S-x)\}.
$$

For simplicity, let $n:=\sum_{i=1}^sn_i$ and $A:=[A_1~ A_2~ \dots ~A_s]$. Then,  problem \eqref{model} can be rewritten as
\begin{equation}\label{model2}
\min\left\{\sum_{i=1}^s f_i(x_i)+\iota_{C}(Ax): x_i\in\mathbb{R}^{n_i}, i\in\mathbb{N}_s\right\},
\end{equation}
where
\begin{equation}\label{def:c}
C:=\{b \}.
\end{equation}

 Now, we are ready to characterize  the solutions of model~\eqref{model}
with the help of \eqref{eq:sub-prox-0} and \eqref{eq:dual-sub}.

\begin{theorem}\label{PropSolu}
Let $f_i \in \Gamma_0(\mathbb{R}^{n_i})$, $A_i$ an $m \times n_i$ matrix for $i\in\mathbb{N}_s$ and  $b\in A (\mathrm{ri}(\mathrm{dom}(\sum_{i=1}^s f_i)))$.  If
$x:=(x_1,x_2,\dots,x_s)\in\mathbb{R}^{n_1}\times\mathbb{R}^{n_2}\times\dots\times\mathbb{R}^{n_s}$ is a solution of problem~\eqref{model}, then
for any $\beta>0$ and $\alpha_i>0$, $i\in\mathbb{N}_s$, there
exists a vector $y\in\mathbb{R}^{m}$ such that
\begin{eqnarray} \label{FixEq1}
x_i&=&\mathrm{prox}_{\frac{\alpha_i}{\beta}f_i}(x_i-\frac{\alpha_i}{\beta}A_i^\top y), ~~ i\in\mathbb{N}_s, \\  \label{FixEq2}
y&=&\mathrm{prox}_{\beta\iota_{C}^*}(y+\beta\sum_{i=1}^s A_ix_i).
\end{eqnarray}
Conversely, if there exist  $\beta>0$, $\alpha_i>0$ for $i\in\mathbb{N}_s$, $x:=(x_1,x_2,\dots,x_s)\in\mathbb{R}^{n_1}\times\mathbb{R}^{n_2}\times\dots\times\mathbb{R}^{n_s}$  and $y\in\mathbb{R}^{m}$
satisfying equations \eqref{FixEq1} and \eqref{FixEq2}, then $x$ is
a solution of problem~\eqref{model}.
\end{theorem}
\begin{proof}\ \
We prove this theorem by applying Fermat's rule that a vector
$x:=(x_1,x_2,\dots,x_s)\in\mathbb{R}^{n_1}\times\mathbb{R}^{n_2}\times\dots\times\mathbb{R}^{n_s}$   is a solution of model~\eqref{model} if and
only if the zero vector is in the subdifferential of the objective
function of model~\eqref{model} evaluated at $x$.

Let $x:=(x_1,x_2,\dots,x_s)\in\mathbb{R}^{n_1}\times\mathbb{R}^{n_2}\times\dots\times\mathbb{R}^{n_s}$    be a solution of model~\eqref{model}. From Theorem 16.37 of \cite{Bauschke-Combettes:11}, the chain rule of the subdifferential holds due to $b\in A (\mathrm {ri}(\mathrm{dom}(\sum_{i=1}^m f_i)))$. Then by
 Fermat's rule we obtain
\begin{equation}\label{zero}
0\in\partial f_i(x_i)+A_i^{\top}\partial \iota_{C}(Ax)
\end{equation}
for $i\in\mathbb{N}_s$.
Thus, there exists $y\in\mathbb{R}^{m}$ such that  $y \in \partial \iota_{C}(Ax)$ and
$-A_i^{\top} y \in \partial f_i(x_i)$ for $i\in\mathbb{N}_s$. The last inclusion implies that for any $\alpha_i>0$, $\beta>0$, $-\frac{\alpha_i}{\beta}A_i^{\top} y \in
\partial (\frac{\alpha_i}{\beta} f_i)(x_i)$. Therefore,
equation \eqref{FixEq1} follows from \eqref{eq:sub-prox-0}. By
\eqref{eq:dual-sub}, from $y \in \partial \iota_{C}(Ax)$, we have that
$Ax \in \partial \iota_{C}^*(y)$.  Hence,  for any $\beta>0$, we obtain that $\beta Ax \in \partial (\beta\iota_{C}^*)(y)$,
which  by \eqref{eq:sub-prox-0} is equivalent to equation
\eqref{FixEq2}.

Conversely, suppose that there exist  $\alpha_i>0$, $\beta>0$, $y\in\mathbb{R}^{m}$ and $x_i\in\mathbb{R}^{n_i}$ for $i\in\mathbb{N}_s$
satisfying the system of fixed-point equations \eqref{FixEq1} and \eqref{FixEq2}. The relation
\eqref{eq:sub-prox-0} ensures that $y \in \partial \iota_{C}(A x)$ and
$- A_i^{\top} y \in \partial f_i(x_i)$.   Clearly, these inclusions together ensure that
the relation \eqref{zero} holds. That is, the zero vector is in the subdifferential
of the objective function at $(x_1, \dots,x_s)$. Again, by  Fermat's rule, $(x_1,\dots,x_s)$ is
a solution of model~\eqref{model}.
\end{proof}

Theorem \ref{PropSolu} characterizes a solution of
problem~\eqref{model} in terms of the system of fixed-point equations
\eqref{FixEq1} and \eqref{FixEq2}. Through out this paper, for problem~\eqref{model}, we assume that $b\in A( \mathrm{ri}(\mathrm{dom}(\sum_{i=1}^s f_i)))$ and it has at least one solution. With these assumptions and by Theorem \ref{PropSolu}, we know that fixed-point equations  \eqref{FixEq1} and \eqref{FixEq2} have at least one solution for any $\alpha_i>0$, $i\in\mathbb{N}_s$ and $\beta>0$. This makes it possible for us to compute a solution of model~\eqref{model} by developing fixed-point iterative schemes.

%

\section{A Two-step Iterative Scheme}\label{sec:algorithm}

We develop in this section a two-step iterative scheme for solving
 optimization problem~\eqref{model} by using the system of
fixed-point equations \eqref{FixEq1} and \eqref{FixEq2}.

We begin with rewriting  equations \eqref{FixEq1} and \eqref{FixEq2}
in a compact form. To this end, we first introduce an operator by
integrating together the $s+1$ proximity operators involved in
equations \eqref{FixEq1} and \eqref{FixEq2}. Specifically, for given
$f_i \in \Gamma_0(\mathbb{R}^{n_i})$, $\iota_{C} \in
\Gamma_0(\mathbb{R}^m)$, $\alpha_i >0$, $\beta>0$, $i\in\mathbb{N}_s$, we define the operator ${\cal T}:=T^{(f_1,\dots, f_s, \frac{\alpha_1}{\beta},\dots, \frac{\alpha_s}{\beta})}_{(\iota_{C},\beta)}: \mathbb{R}^{n_1}\times\dots\times\mathbb{R}^{n_s} \times \mathbb{R}^m \rightarrow
\mathbb{R}^{n_1}\times\dots\times\mathbb{R}^{n_s} \times \mathbb{R}^m$ at a vector $v:=(x_1,\dots,x_s,y) \in
\mathbb{R}^{n_1}\times\dots\times\mathbb{R}^{n_s} \times \mathbb{R}^m$ as follows:
\begin{equation}\label{eq:def-T}
{\cal T}(v):=(\mathrm{prox}_{\frac{\alpha_1}{\beta}f_1} (x_1),\dots,\mathrm{prox}_{\frac{\alpha_s}{\beta}f_s} (x_s),
\mathrm{prox}_{\beta\iota_{C}^*}(y)).
\end{equation}
Operator ${\cal T}$ couples all the proximity operators
$\mathrm{prox}_{\frac{\alpha_i}{\beta}f_i}$, $i\in\mathbb{N}_s$ and $\mathrm{prox}_{\beta\iota^*_{C}}$. In the following lemma, we show that the operator ${\cal T}$ is the proximity
operator of a new convex function
\begin{equation}\label{def:phi}
 \Phi(v):=\sum_{i=1}^sf_i(x_i)+\iota^*_{C}(y)
 \end{equation} for
$v:=(x_1,\dots,x_s,y)$ with respect to the
matrix
\begin{equation}
\label{def:r}
R:=\mathrm{diag}\left(\frac{\beta}{\alpha_1}{\bf{1}}_{n_1},\dots,\frac{\beta}{\alpha_s}{\bf{1}}_{n_s}, \frac{1}{\beta}{\bf{1}}_m\right),
\end{equation}
where  ${\bf{1}}_d$ (resp. ${\bf{0}}_d$) is a $d$-dimensional vector with $1$ (resp. $0$) as its components for any $d\in\mathbb{N}$.

\begin{lemma}\label{lemma:proc-comb}
If operator ${\cal T}$ is defined by \eqref{eq:def-T}, then
${\cal T}$ is the proximity operator of the function $\Phi$ with
respect to the matrix $R$, that is, ${\cal T}=\mathrm{prox}_{\Phi,
R}$.
\end{lemma}

Here we omit the proof since one can complete it by referring to Lemma 3.1 of \cite{Li-Shen-Xu-Zhang:AiCM:14}.
By Lemma~\ref{lemma:proc-comb}, we know that the operator ${\cal T}$ is
firmly non-expansive with respect to the matrix $R$.
Let \begin{equation}
\label{def:P}
P:=\mathrm{diag}\left(\frac{\beta}{\alpha_1}{\bf{1}}_{n_1},\dots,\frac{\beta}{\alpha_s}{\bf{1}}_{n_s}\right).
\end{equation}With the help of the above notation, equations~\eqref{FixEq1} and
\eqref{FixEq2} can be reformulated in a compact form
\begin{equation}\label{Fix}
v=({\cal T}\circ E)(v),
 \end{equation}
 where
 \begin{equation}\label{def:e}
 E:=\begin{bmatrix}
 I&-P^{-1}A^\top\\
 \beta A&I
 \end{bmatrix}.
 \end{equation}
Theorem~\ref{PropSolu} together with equation \eqref{Fix} indicates that finding a solution of problem ~\eqref{model} essentially amounts to computing a fixed-point of the operator ${\cal T} \circ E$. As discussed at the end of Section \ref{sec:characterization}, the operator ${\cal T} \circ E$ has at least one fixed-point. We next focus on developing efficient iterative schemes for finding a fixed-point of the operator. As shown in \cite{Li-Shen-Xu-Zhang:AiCM:14}, the matrix $E$ is not nonexpansive due to the fact that $\|E\|_2>1$. Therefore, a simple fixed-point iteration $v^{k+1}=({\cal T}\circ E)(v^k)$ for a given initial guess $v^0$, may not yield a convergent sequence $\{v^k: k \in
\mathbb{N}\}$, where $\mathbb{N}$ is the set of all natural numbers.

Our idea is to split the expansive matrix $E$ into several terms, as in \cite{Li-Shen-Xu-Zhang:AiCM:14} and in \cite{Li-Zhang:ACHA:2015}. Here, we split $E$ as
\begin{equation}\label{eq:splitE}
E=(E-R^{-1}M_0)+R^{-1}M_1+R^{-1}M_2,
\end{equation}
where $M_i\in\mathbb{R}^{(n+m)\times(n+m)}$ for $i=0,1,2$ and $M_0=M_1+M_2$.
 Accordingly,  equation \eqref{Fix} is equivalent to $$v={\cal{T}}((E-R^{-1}M_0)v+R^{-1}M_1v+R^{-1}M_2v).$$ Thus, we propose the following two-step iterative scheme:
\begin{equation}\label{eq:general-ite}
v^{k+1}={\cal T}\left((E-R^{-1}M_0)v^{k+1}+R^{-1}M_1 v^k+R^{-1}M_2v^{k-1}\right).
\end{equation}
We point out here that although iterative scheme \eqref{eq:general-ite} is an implicit scheme for the whole vector $v$, it becomes  explicit by choosing $M_0$ satisfying that  $E-R^{-1}M_0$ is a strictly upper triangular or lower triangular matrix. Further, we assume that there exists a unique $v^{k+1}$ satisfying \eqref{eq:general-ite} for any $v^{k}, v^{k-1}\in\mathbb{R}^{n+m}$ in the rest of this paper.  We shall choose matrices  $M_0, M_1, M_2$ in the next section so that  iterative scheme \eqref{eq:general-ite}  converges.

To close this section, we remark that  when
$M_2=0$ (in this case, $M_0=M_1$), the two-step iterative scheme \eqref{eq:general-ite}
reduces to a one-step iterative scheme
\begin{equation}\label{eq:oneiter}
v^{k+1}={\cal T}\left((E-R^{-1}M_0)v^{k+1}+R^{-1}M_0 v^{k}\right).
\end{equation}
Many efficient algorithms can be obtained from
\eqref{eq:oneiter} by specifying the matrix $M_0$. The reader is referred to \cite{Li-Shen-Xu-Zhang:AiCM:14} for details.

\section{Convergence Analysis of the Proposed Iterative Scheme}\label{sec:converge}
In this section, we study the convergence of  iterative scheme \eqref{eq:general-ite}.  By applying the notion of weakly firmly nonexpansive operators and Condition-M, which were first introduced  in \cite{Li-Shen-Xu-Zhang:AiCM:14}, we prove  that  if the matrices $M_0, M_1, M_2$  satisfy Condition-M,  then the sequence $\{v^{k}: k \in \mathbb{N}\}$ generated from iterative scheme  \eqref{eq:general-ite} converges to a solution of equation~\eqref{Fix}. Hence, the sequence $\{x^k: k\in\mathbb{N}\}$ converges to a solution of model~\eqref{model}.

We begin with rewriting iterative scheme \eqref{eq:general-ite} in an explicit way. To this end, we introduce  $\mathcal{M}:=\{M_0, M_1, M_2\}$. We also define
$T_{\mathcal{M}}:\mathbb{R}^{n+m}\times\mathbb{R}^{n+m}\rightarrow \mathbb{R}^{n+m}$, at $(u_1, u_2)\in \mathbb{R}^{n+m}\times\mathbb{R}^{n+m}$, as $w=T_\mathcal{M}(u_1, u_2)$ with $w$ satisfying
\begin{equation}\label{def:tm}
 w=\mathcal{T}((E-R^{-1}M_0)w+R^{-1}M_1u_1+R^{-1}M_2u_2).
 \end{equation}
The operator $T_\mathcal{M}$ is well-defined if the corresponding set $\mathcal{M}$ is carefully chosen. Here, the word ``well-defined'' means that there exists a unique $w\in\mathbb{R}^{n+m}$ satisfying \eqref{eq:general-ite} for any $(u_1, u_2)\in\mathbb{R}^{n+m}\times\mathbb{R}^{n+m}$. With the help of $\mathcal{M}$ and $T_\mathcal{M}$, \eqref{eq:general-ite} can be rewritten as
\begin{equation}\label{eq:gene-ite2}
v^{k+1}=T_\mathcal{M}(v^k, v^{k-1}).
\end{equation}

Now, we recall the notion of weakly firmly nonexpansive operators and Condition-M, which were introduced in \cite{Li-Shen-Xu-Zhang:AiCM:14}.
\begin{definition}[Weakly Firmly Nonexpansive]\label{def:wfn}
We say an operator $T: \mathbb{R}^{2d}\rightarrow \mathbb{R}^d$ is weakly firmly nonexpansive with respect to $\mathcal{M}$, if for any $(u_i, w_i, z_i)\in\mathbb{R}^d\times\mathbb{R}^d\times\mathbb{R}^d$ satisfying $z_i=T(u_i, w_i)$ for $i=1,2$, there holds
$$
\langle z_2-z_1, M_0(z_2-z_1)\rangle\le \langle z_2-z_1, M_1(u_2-u_1)+M_2(w_2-w_1)\rangle.
$$
\end{definition}
Next we  describe the definition of Condition-M.
\begin{definition}[Condition-M]\label{def}
We say a set $\mathcal{M}:=\{M_0, M_1, M_2\}$ of $d\times d$ matrices satisfies Condition-M, if the following three hypotheses are satisfied:
\begin{itemize}
\item [(i)] $M_0=M_1+M_2$,
\item [(ii)] $H:=M_0+M_2$ is in $\mathbb{S}^d_+$,
\item [(iii)] $\|H^{-\frac{1}{2}}M_2H^{-\frac{1}{2}}\|_2<\frac{1}{2}$.
\end{itemize}
\end{definition}

We also need to review a  property of weakly firmly nonexpansive operators established in \cite{Li-Shen-Xu-Zhang:AiCM:14}.
\begin{theorem}\label{thm:weakConve}
Suppose that the operator $T:\mathbb{R}^{2d}\rightarrow \mathbb{R}^d$ is weakly firmly nonexpansive with respect to $\mathcal{M}:=\{M_0, M_1, M_2\}$ with $\mathrm{dom}(T)=\mathbb{R}^{2d}$ and the set of fixed-points of $T$ is nonempty. Let the sequence $\{w^k: k\in\mathbb{N}\}$ be generated by $w^{k+1}=T(w^k, w^{k-1})$ for any given $w^0, w^1\in\mathbb{R}^d$.
If $\mathcal{M}$ satisfies Condition-M, then $\{w^k: k\in\mathbb{N}\}$ converges. In addition, if $T$ is continuous, then $\{w^k: k\in\mathbb{N}\}$ converges to a fixed-point of $T$.
\end{theorem}

By the above theorem, in order to ensure   convergence of iterative scheme \eqref{eq:gene-ite2}, it  suffices to  prove $T_\mathcal{M}$ defined by \eqref{def:tm} is weakly firmly nonexpansive and continuous. We show it in the next proposition. Before doing this, we define a skew-symmetric matrix $S_{A}$ for an $m\times n$ matrix $A$ as
\begin{equation}\label{def:sa}
S_{A}:=\begin{bmatrix}
0&-A^\top \\
A&0
\end{bmatrix}.
\end{equation}
 Then, $E=I+R^{-1}S_A$.

\begin{proposition}\label{prop:tmWFN}
Let $f_i \in \Gamma_0(\mathbb{R}^{n_i})$, $\alpha_i>0$ for $i\in\mathbb{N}_s$ and $\beta>0$.  Let $\mathcal{M}:=\{M_0, M_1, M_2\}$ be a set of $(n+m) \times (n+m)$ matrices and  $T_\mathcal{M}$  be defined by \eqref{def:tm}. If $T_\mathcal{M}$ is well-defined, then
 \begin{itemize}
 \item [(i)] $T_\mathcal{M}$ is weakly firmly nonexpansive with respect to $\mathcal{M}$,
 \item [(ii)] $T_\mathcal{M}$ is continuous.
 \end{itemize}
\end{proposition}
\begin{proof}
We first prove Item (i). It follows from the definition of $T_\mathcal{M}$ that for any  $(u_i, w_i, z_i)\in\mathbb{R}^{n+m}\times\mathbb{R}^{n+m}\times\mathbb{R}^{n+m}$ satisfying $z_i=T_\mathcal{M}(u_i, w_i)$, for $i=1,2$, there holds
$$
z_i=\mathcal{T}((E-R^{-1}M_0)z_i+R^{-1} M_1 u_i+R^{-1}M_2w_i).
$$
According to Lemma \ref{lemma:proc-comb}, $\mathcal{T}$ is firmly nonexpansive with respect to $R$. Thus, we observe that
$$
\|z_2-z_1\|_R^2\le \langle z_2-z_1, (RE-M_0)(z_2-z_1)+M_1(u_2-u_1)+M_2(w_2-w_1)\rangle.
$$
Since $RE=R+S_A$ and $S_A$ is skew-symmetric, we have
$$
\langle z_2-z_1, M_0(z_2-z_1)\rangle\le\langle z_2-z_1, M_1(u_2-u_1)+M_2(w_2-w_1)\rangle.
$$
From Definition \ref{def:wfn}, we get Item (i).

We next prove Item (ii). From the definition of $T_\mathcal{M}$, for any  sequence  $\{(u^k,w^k,z^k)\in\mathbb{R}^{n+m}\times \mathbb{R}^{n+m}\times\mathbb{R}^{n+m}: k\in\mathbb{N}\}$  satisfying $z^k=T_\mathcal{M}(u^k, w^k)$ and converging to $(u,w,z)$,  we have that
$
z^k=\mathcal{T}((E-M_0)z^k+R^{-1}M_1 u^k+R^{-1}M_2 w^k).
$
This with the continuity of $\mathcal{T}$ implies that $z=\mathcal{T}((E-M_0)z+R^{-1}M_1 u+R^{-1}M_2 w).$ Thus, $z=T_\mathcal{M}(u,w)$, proving Item (ii).
\end{proof}

We are now ready to prove convergence of the sequence generated from  iterative scheme \eqref{eq:general-ite}.

\begin{theorem}\label{thm:itr_con}
Let $f_i \in \Gamma_0(\mathbb{R}^{n_i})$, $\alpha_i>0$ for $i\in\mathbb{N}_s$ and $\beta>0$. Let $\mathcal{T}$ and $E$ be defined as \eqref{eq:def-T}  and \eqref{def:e} respectively,  $\mathcal{M}:=\{M_0, M_1, M_2\}$ be a set of $(n+m) \times (n+m)$ matrices and  $T_\mathcal{M}$  be defined by \eqref{def:tm}.  Let $\{v^k: k\in\mathbb{N}\}$ be generated by \eqref{eq:general-ite}  for given points $v^0, v_1$. Suppose that $T_\mathcal{M}$ is well-defined.
If $\mathcal{M}$ satisfies Condition-M, then the sequence $\{v^k: k\in\mathbb{N}\}$ converges to a fixed-point of $\mathcal{T}\circ E$, and $\{x^k: k\in\mathbb{N}\}$ converges to a solution of problem \eqref{model}.
\end{theorem}
\begin{proof}\ \
By the definition of $T_\mathcal{M}$, operators $T_\mathcal{M}$ and $\mathcal{T}\circ E$ share the same set of fixed-points. By Proposition \ref{prop:tmWFN}, the operator $T_\mathcal{M}$ is weakly firmly non-expansive with respect to $\mathcal{M}$ and continuous. Therefore, Theorem \ref{thm:weakConve} ensures that the sequence $\{v^k: k\in\mathbb{N}\}$  converges to a fixed-point of $T_\mathcal{M}$.  By Proposition \ref{PropSolu}, the sequence $\{x^k: k\in\mathbb{N}\}$ converges to a solution of problem \eqref{model}.
\end{proof}

Theorem~\ref{thm:itr_con} shows that convergence of  iterative scheme \eqref{eq:general-ite} relies completely on whether the matrices set $\mathcal{M}$ used in scheme \eqref{eq:general-ite} satisfies Condition-M. We will develop  in Section \ref{sec:specificAlg} specific convergent algorithms by generating  sets of $\{M_0, M_1, M_2\}$ satisfying Condition-M.

\section{Convergence Rate of the Proposed Two-step Iterative Scheme}
\label{sec:ConvRate}
In this section, we study the convergence rate of the proposed fixed-point iterative scheme \eqref{eq:general-ite}. We show that the proposed algorithm has $O(\frac{1}{k})$ convergence rate in the ergodic sense and the sense of the partial primal-dual gap.

\subsection{Ergodic $O(\frac{1}{k})$ Rate}
We first study the convergence rate of the proposed algorithm \eqref{eq:general-ite} in the ergodic sense. We prove in this subsection that the proposed iterative scheme \eqref{eq:general-ite} has $O(\frac{1}{k})$ convergence in the ergodic sense. To this end, we first review a lemma presented in \cite{Shi-Ling-Wu-Yin:siamjopt:2015}.
\begin{lemma}\label{lema:running}
If a sequence $\{a^k: k\in\mathbb{N}\}$ satisfies: $a^k\ge 0$ and $\sum_{i=1}^{+\infty}a^i< +\infty$, then
\begin{itemize}
\item [(i)] $\frac{1}{k}\sum_{i=1}^k {a^i}=O(\frac{1}{k})$,
\item [(ii)] $\min_{i\le k}\{a^i\}=o(\frac{1}{k})$.
\end{itemize}
\end{lemma}

The main results of this subsection are presented in the next theorem.
\begin{theorem}
Let $f_i \in \Gamma_0(\mathbb{R}^{n_i})$ and $A_i$ an $m \times n_i$ matrix for $i\in\mathbb{N}_s$.
Let  $\alpha_i>0$ for $i\in\mathbb{N}_s$ and $\beta>0$. Let $\mathcal{T}$ and $E$ be defined as \eqref{eq:def-T}  and \eqref{def:e} respectively.
Let the sequence $\{v^k:k\in\mathbb{N}\}$ be generated from \eqref{eq:general-ite} for any given $v^0, v^1\in\mathbb{R}^{n+m}$.  Suppose that $T_\mathcal{M}$ is well-defined. If $\mathcal{M}$ satisfies Condition-M, then
\begin{itemize}
\item [(i)] the sequence $\{v^k: k\in\mathbb{N}\}$ has $O(\frac{1}{k})$ convergence in the ergodic sense, that is
\begin{equation}
\frac{1}{k}\sum_{i=1}^k\|v^{i+1}-v^i\|_2^2=O(\frac{1}{k}),
\end{equation}
\item [(ii)]  the running minimal of progress, $\min_{i\le k}\{\|v^{i+1}-v^i\|_2^2\}$, has $o(\frac{1}{k})$ convergence.
\end{itemize}
\end{theorem}
\begin{proof}
By Lemma \ref{lema:running}, we only need to prove \begin{equation}\label{eq:sumable}
\sum_{i=1}^{+\infty}\|v^{i+1}-v^i\|_2^2<+\infty.
\end{equation}

By the definition of $T_\mathcal{M}$, the sequence $\{v^k: k\in\mathbb{N}\}$ generated from \eqref{eq:general-ite} can also be generated by \eqref{eq:gene-ite2} for the same given $v^0, v^1$. Since $T_\mathcal{M}$ is weakly firmly nonexpansive with respect to $\mathcal{M}$ and $\mathcal{M}$ satisfies Condition-M, by Lemma 4.4 of \cite{Li-Shen-Xu-Zhang:AiCM:14}, we have for any $k\ge 3$ that
\begin{equation}\label{eq:eq8}
\|e^{k}\|_H^2\le 2\|e^2\|_H^2+2\|\widetilde {M}\|_2^2\|r^2\|_H^2-2\langle e^2, M_2r^2\rangle-(\frac{1}{2}-2\|\widetilde M\|_2^2)\sum_{i=2}^{k-1}\|r^{i+1}\|_H^2,
\end{equation}
where $e^i:=v^i-v$ for $v$ a fixed-point of $T_\mathcal{M}$, $r^i:=v^{i}-v^{i-1}$ and $\widetilde M:= H^{-1/2}M_2H^{-1/2}$. By (iii) of Condition-M, we have $\frac{1}{2}-2\|\widetilde M\|_2^2>0$. Then \eqref{eq:sumable} is obtained immediately from  \eqref{eq:eq8} and the fact that $H\in \mathbb{S}_+^{n+m}$.
\end{proof}
\subsection{Partial Primal-dual Gap $O(\frac{1}{k})$ Convergence Rate }
In this subsection, we study the convergence rate of the proposed iterative algorithm \eqref{eq:general-ite} in the  sense of the partial primal-dual gap. We prove that   iterative scheme \eqref{eq:general-ite} has $O(\frac{1}{k})$ convergence rate in the sense of the partial primal-dual gap.

We first introduce the notion of the partial primal-dual gap for  convex problem \eqref{model}. To this end, we review the primal-dual formulation of problem \eqref{model}, that is
\begin{equation}\label{eq:pri-dualProb}
\min\{\max\{\sum_{i=1}^s f_i(x_i)-\iota_{C}^*(y)+\langle Ax, y\rangle: y\in\mathbb{R}^m\}: x_i\in\mathbb{R}^{n_i}, i\in\mathbb{N}_s\}.
\end{equation}
One can refer to \cite{Bauschke-Combettes:11} for more details. For two bounded sets $B_1\subseteq\mathbb{R}^n$ and $B_2\subseteq\mathbb{R}^m$, the partial primal-dual gap for  problem \eqref{model} at point $v:=(x_1,\dots,x_s, y)\in\mathbb{R}^{n_1}\times\dots\times\mathbb{R}^{n_s}\times \mathbb{R}^m$  is defined as
\begin{equation}
\begin{array}{rcl}
\mathcal{G}_{B_1\times B_2}(v)&:=&\max\{\sum_{i=1}^sf_i(x_i)-\iota_{C}^*(y')+\langle Ax, y'\rangle: y'\in B_2\}\\
&&-\min\{\sum_{i=1}^sf_i(x_i')-\iota_{C}^*(y)+\langle  Ax', y\rangle: x'\in B_1\}.
\end{array}
\end{equation}
 We refer to \cite{Chambolle-Pock:JMIV11} for more details on the partial primal-dual gap.

In order to analyze the convergence rate of iterative scheme \eqref{eq:general-ite}, we define $G:\mathbb{R}^{n+m}\times\mathbb{R}^{n+m}\rightarrow\overline{\mathbb{R}}$ by
\begin{equation}\label{eq:G}
G(v, v'):=\Phi(v)-\Phi(v')+\langle v', S_{A} v\rangle,
\end{equation}
where $\Phi$ and $S_{A}$ are defined as \eqref{def:phi} and \eqref{def:sa} respectively. For $(v, v')\in\mathbb{R}^{n+m}\times\mathbb{R}^{n+m}$, where $v':=(x_1',\dots, x_s', y')\in\mathbb{R}^{n_1}\times\dots\times\mathbb{R}^{n_s}\times\mathbb{R}^m$ and $v:=(x_1,\dots,x_s, y)\in\mathbb{R}^{n_1}\times\dots\times\mathbb{R}^{n_s}\times\mathbb{R}^m$, one can check that \eqref{eq:G}  is equivalent to
$$
G(v, v')=\sum_{i=1}^sf_i(x_i)-\iota_{C}^*(y')+\langle Ax, y'\rangle-(\sum_{i=1}^sf_i(x_i')-\iota_{C}^*(y)+\langle  Ax', y\rangle).
$$
Therefore, in order to analyze the partial primal-dual gap at point $v:=(x, y)\in\mathbb{R}^n\times\mathbb{R}^m$, we only need to estimate the upper bound of $G(v, v')$ for $v'\in B_1\times B_2$.
The next lemma presents an important  estimation of $G(v^{k+1}, v)$ for any $v\in\mathbb{R}^{n+m}$.

\begin{lemma}\label{lema: InitialPPriamalDualIne}
Let $f_i \in \Gamma_0(\mathbb{R}^{n_i})$, $A_i$ an $m \times n_i$ matrix,
   $\alpha_i>0$ for $i\in\mathbb{N}_s$ and $\beta>0$.  Let $\mathcal{T}$ and $E$ be defined as \eqref{eq:def-T}  and \eqref{def:e} respectively, $\mathcal{M}:=\{M_0, M_1, M_2\}$ be a set of $(n+m) \times (n+m)$ matrices.
Let $\{v^k:=(x^k, y^k)\in\mathbb{R}^n\times\mathbb{R}^m:  k\in\mathbb{N}\}$ be generated from iterative scheme \eqref{eq:general-ite}. For all $v\in\mathbb{R}^{n+m}$ there holds
\begin{equation}\label{eq:parDualIne}
G( v^{k+1}, v)\le \langle M_1v^k+M_2v^{k-1}-M_0v^{k+1}, v^{k+1}-v\rangle.
\end{equation}
\end{lemma}
\begin{proof}
From iterative scheme \eqref{eq:general-ite}, Lemma \ref{lemma:proc-comb} and \eqref{eq:sub-prox-0}, we have
$$
R(E-R^{-1}M_0-I)v^{k+1}+M_1v^k+M_2v^{k-1}\in\partial \Phi( v^{k+1}).
$$
Due to $E=I+R^{-1}S_{A}$, we obtain that
$$
-M_0v^{k+1}+M_1v^k+M_2v^{k-1}+S_{A}v^{k+1}\in\partial \Phi( v^{k+1}).
$$
By  the definition of subdifferential and the convexity of $\Phi$, we have for any $v\in\mathbb{R}^{n+m}$ that
$$
\Phi( v^{k+1})+\langle M_1v^k+M_2v^{k-1}-M_0v^{k+1}, v-v^{k+1}\rangle+\langle v-v^{k+1}, S_{A} v^{k+1}\rangle\le\Phi(v).
$$
Since $S_{A}$ is skew-symmetric, the above inequality is equivalent to
$$
\Phi( v^{k+1})+\langle v, S_{A} v^{k+1}\rangle-\Phi(v)\le \langle M_1v^k+M_2v^{k-1}-M_0v^{k+1}, v^{k+1}-v\rangle.
$$
Then, we obtain \eqref{eq:parDualIne} immediately by   the definition of $G$.
\end{proof}

We next study the partial primal-dual gap at $\bar v_K:=\frac{\sum_{k=2}^{K+1} v^k}{K}$.

\begin{lemma}\label{lema: PPriamalDualIne}
Let $\{v^k:  k\in\mathbb{N}\}$  be generated from iterative scheme \eqref{eq:general-ite}. Under the same assumptions  of Lemma \ref{lema: InitialPPriamalDualIne}, if $\mathcal{M}$ satisfies Condition-M,  then for all $v\in\mathbb{R}^{n+m}$ there holds
\begin{equation}\label{eq:AvepDuaIne}
G(\bar v_K, v)\le \frac{\frac{3}{4}\|v^1-v\|_H^2+\frac{1}{2}\|v^1-v^0\|_H^2}{K},
\end{equation}
where $H:=M_0+M_2$ and $\bar v_K:=\frac{\sum_{k=2}^{K+1} v^k}{K}$.
\end{lemma}
\begin{proof}
For simplicity, we define for $k\in\mathbb{N}$, $
e^k:=v^k-v, r^k:=v^k-v^{k-1}.
$
By Lemma \ref{lema: InitialPPriamalDualIne} and Item (i) of Condition-M, we have
$$
G(v^{k+1}, v)\le \langle M_1e^k+M_2e^{k-1}, e^{k+1}\rangle-\langle M_0e^{k+1}, e^{k+1}\rangle.
$$
Using $H:=M_0+M_2$ and $M_0=M_1+M_2$, the above inequality implies that
\begin{equation}\label{eq:eq9}
G(v^{{k+1}}, v)\le D_1+D_2,
\end{equation}
where $D_1:=-\|e^{k+1}\|_H^2+\langle e^{k+1}, He^k\rangle$ and $D_2:=\langle e^{k+1}, M_2e^{k+1}-2M_2e^k+M_2e^{k-1}\rangle$.
By the relationship $r^{k+1}=e^{k+1}-e^k$ and $\langle a, Hb\rangle=\frac{1}{2}(\|a\|_H^2+\|b\|_H^2-\|a-b\|_H^2)$ for $a, b\in\mathbb{R}^{n+m}$, we obtain that \begin{equation}
\label{eq:eq10}
D_1=\frac{1}{2}(-\|e^{k+1}\|_H^2+\|e^k\|_H^2-\|r^{k+1}\|_H^2).
\end{equation}
We also have
\begin{eqnarray}\label{eq:eq11}
D_2&=&\langle e^{k+1}, M_2(r^{k+1}-r^k)\rangle\nonumber\\
&=&\langle e^{k+1}, M_2r^{k+1}\rangle-\langle r^{k+1}, M_2r^{k}\rangle-\langle e^{k}, M_2r^{k}\rangle,
\end{eqnarray}
where the first equality is obtained by  the relationship $r^{k}=e^{k}-e^{k-1}$ and the second equality holds due to $e^{k+1}=r^{k+1}+e^k$. Let $\widetilde M:=(H^\dagger)^{1/2}M_2(H^\dagger)^{1/2}$. Then it follows that for any $a>0$,
\begin{equation}\label{eq:eq12}
|\langle r^{k+1}, M_2r^k\rangle|\le \frac{a}{2}\|r^{k+1}\|_H^2+\frac{\|\widetilde M\|_2^2}{2a}\|r^k\|_H^2.
\end{equation}
 Thus, by \eqref{eq:eq9}, \eqref{eq:eq10}, \eqref{eq:eq11} and \eqref{eq:eq12}, we have
 \begin{eqnarray}\label{eq:eq13}
 G(v^{k+1}, v)&\le& \frac{1}{2}(-\|e^{k+1}\|_H^2+\|e^k\|_H^2)
 -\frac{1}{2}(1-a-\frac{\|\widetilde M\|_2^2}{a})\|r^{k+1}\|_H^2\nonumber\\
 &&+\frac{\|\widetilde M\|_2^2}{2a}(-\|r^{k+1}\|_H^2+\|r^{k}\|_H^2)
 +\langle e^{k+1}, M_2r^{k+1}\rangle-\langle e^k, M_2r^k\rangle.
 \end{eqnarray}

Summing inequality \eqref{eq:eq13} from $k=1$ to $k=K$, we have
\begin{eqnarray}\label{eq:eq14}
\sum_{k=1}^{K} G(v^{k+1}, v)&\le& \frac{1}{2}(-\|e^{K+1}\|_H^2+\|e^1\|_H^2)
-\frac{1}{2}(1-a-\frac{\|\widetilde M\|_2^2}{a})\sum_{k=2}^{K+1}\|r^{k}\|_H^2 \nonumber \\
& & +\frac{\|\widetilde M\|_2^2}{2a}(-\|r^{K+1}\|_H^2+\|r^1\|_H^2)
+\langle e^{K+1}, M_2r^{K+1}\rangle-\langle e^1, M_2r^1\rangle.
 \end{eqnarray}
By applying  $$|\langle e^{k}, M_2r^k\rangle|\le \frac{a}{2}\|e^{k}\|_H^2+\frac{\|\widetilde M\|_2^2}{2a}\|r^k\|_H^2$$ for $k=K+1$ and $k=1$ to the last two terms of \eqref{eq:eq14}, we obtain that
\begin{eqnarray}\label{eq:eq15}
\sum_{k=1}^{K} G(v^{k+1}, v)&\le& \frac{1}{2}(-(1-a)\|e^{K+1}\|_H^2+(1+a)\|e^1\|_H^2)\nonumber\\
&&-\frac{1}{2}(1-a-\frac{\|\widetilde M\|_2^2}{a})\sum_{k=2}^{K+1}\|r^{k}\|_H^2
+\frac{\|\widetilde M\|_2^2}{a}\|r^1\|_H^2.
\end{eqnarray}
Setting $a=\frac{1}{2}$, it follows that $1-a-\frac{\|\widetilde M\|_2^2}{a}>0$ due to  (iii) of Condition-M. This together with \eqref{eq:eq15} yields
\begin{equation}\label{eq:SumIne}
\frac{1}{K}\sum_{k=1}^{K} G(v^{k+1}, v)\le \frac{\frac{3}{4}\|e^1\|_H^2+\frac{1}{2}\|r^1\|_H^2}{K}.
\end{equation}
Since $G(\cdot, v)$ is convex, we conclude that $G(\bar v_K, v)\le \frac{1}{K}\sum_{k=1}^{K}G(v^{k+1}, v)$, which together with inequality \eqref{eq:SumIne} implies \eqref{eq:AvepDuaIne}.
\end{proof}

Now, we are ready to present the partial primal-gap convergence rate of the proposed algorithm \eqref{eq:general-ite} in the next theorem.
\begin{theorem}
Let $\{v^k:  k\in\mathbb{N}\}$  be generated from iterative scheme \eqref{eq:general-ite}. Under the same assumptions  of Lemma \ref{lema: InitialPPriamalDualIne}, if $\mathcal{M}$ satisfies Condition-M,  then iterative scheme   \eqref{eq:general-ite} has  $O(\frac{1}{K})$ convergence rate in the partial primal-dual gap sense, that is
$$
\mathcal{G}_{B_1\times B_2}(\bar v_K)=O(\frac{1}{K}),
$$
where $B_1\subseteq \mathbb{R}^n$ and $B_2\subseteq\mathbb{R}^m$ are bounded and $\bar v_K:=\frac{1}{K}\sum_{k=2}^{K+1} v^k$.
\end{theorem}
\begin{proof}
This is a direct consequence of Lemma \ref{lema: PPriamalDualIne} and the boundedness of sets $B_1$ and $B_2$.
\end{proof}
\section{Specific Algorithms}\label{sec:specificAlg}
In this section, we derive several specific two-step algorithms from the iterative scheme \eqref{eq:general-ite} by choosing specific sets of $(n+m) \times (n+m)$ matrices $\mathcal{M}:=\{M_0,M_1,M_2\}$ which satisfy Condition-M.
\subsection{First-order primal-dual Algorithms}
In this subsection, we design a class of explicit one-step algorithms,  which only utilize the vectors of the current step to update the vectors of the next step. In such case,  $M_2=0$ and Condition-M reduces to $M_0=M_1$ and $M_0\in\mathbb{S}_+^{n+m}$.

We begin with constructing $M_0$. If the matrix $E-R^{-1}M_0$ is strictly upper or lower triangular, then the resulting algorithms will be  explicit. By \eqref{def:e}, $M_0=M_1$ can be chosen as $Z_1$ or $Z_2$  with
$$Z_1:=
 \begin{bmatrix}
 P&-A^\top\\
 -A& \frac{1}{\beta}I
 \end{bmatrix},
 Z_2=\begin{bmatrix}
 P&A^\top\\
 A& \frac{1}{\beta}I
 \end{bmatrix},
 $$
where $P$ is defined by \eqref{def:P}.
By simple calculations, one can obtain that $\iota_C^*(\cdot)=\langle \cdot, b\rangle$ and thus $\mathrm{prox}_{\beta\iota_C^*}(y)=y-\beta b$ for $y\in\mathbb{R}^m$. Then,  iterative scheme \eqref{eq:general-ite} with respect to $Z_1$ and $Z_2$ become, respectively,
\begin{equation}
\label{eq:al1}
\begin{cases}
 x_i^{k+1}=\mathrm{prox}_{\frac{\alpha_i}{\beta}f_i}(x_i^k-\frac{\alpha_i}{\beta}A_i^\top y^k),~~i\in\mathbb{N}_s,\\
 y^{k+1}=y^k+\beta(\sum_{i=1}^s A_i(2x_i^{k+1}-x_i^k)-b),
 \end{cases}
 \end{equation}
 and
 \begin{equation}\label{eq:al2}
 \begin{cases}
y^{k+1}=y^k+\beta(\sum_{i=1}^s A_i x_i^k-b),\\
 x_i^{k+1}=\mathrm{prox}_{\frac{\alpha_i}{\beta}f_i}(x_i^k-\frac{\alpha_i}{\beta}A_i^\top (2y^{k+1}-y^k)),~~i\in\mathbb{N}_s.
  \end{cases}
  \end{equation}

We note that, algorithms \eqref{eq:al1} and \eqref{eq:al2} are actually special cases of the one-step first-order primal-dual algorithm \cite{Chambolle-Pock:JMIV11,Esser-Zhang-Chan:1order-PD-09,Li-Shen-Xu-Zhang:AiCM:14}, which solves the following optimization problem  \begin{equation}
\label{eq:eq17}
\min\{f(x)+g(Ax): x\in\mathbb{R}^n\}
 \end{equation}
with $f\in\Gamma_0(\mathbb{R}^n)$, $g\in\Gamma_0(\mathbb{R}^m)$ and $A$ an $m\times n$ matrix. Here, if we set
  $x:=(x_1,\dots,x_s)$,  $g:=\iota_{C}$ and $f:\mathbb{R}^{n_1}\times\dots\times\mathbb{R}^{n_s}\rightarrow\overline{\mathbb{R}}$ defined at $x$ as $f(x):=\sum_{i=1}^s f_i(x_i)$, then problem \eqref{eq:eq17} is exactly the optimization problem \eqref{model}.
 Clearly, algorithms \eqref{eq:al1} and \eqref{eq:al2} are  special cases of the one-step first-order primal-dual algorithm \cite{Chambolle-Pock:JMIV11,Li-Shen-Xu-Zhang:AiCM:14} by the fact that $
 \mathrm{prox}_{f, P}(x)=(\mathrm{prox}_{\frac{\alpha_1}{\beta}f_1}(x_1), \dots, \mathrm{prox}_{\frac{\alpha_s}{\beta}f_s}(x_s))$.
The corresponding convergence results are presented in the following theorem.
\begin{theorem}
 Let $f_i \in \Gamma_0(\mathbb{R}^{n_i})$, $A_i$ an $m \times n_i$ matrix, 
 $\alpha_i>0$ for $i\in\mathbb{N}_s$ and $\beta>0$.   Let the sequence $\{(x_1^k, \dots, x_s^k, y^k):k\in\mathbb{N}\}$ generated from \eqref{eq:al1} or \eqref{eq:al2} for any $(x_1^0, \dots, x_s^0, y^0)\in\mathbb{R}^n\times\mathbb{R}^m$. If $\|AQ\|_2<1$, where $Q:=\mathrm{diag}(\sqrt{\alpha_1}{\bf{1}}_{n_1},\dots,\sqrt{\alpha_s}{\bf{1}}_{n_s})$, then $\{(x_1^k, \dots, x_s^k, y^k):k\in\mathbb{N}\}$ converges and the sequence $\{x^k:k\in\mathbb{N}\}$ converges to a solution of problem \eqref{model}.
 \end{theorem}

 We omit the proof since it can be obtained immediately by applying Lemma 6.2 in \cite{Li-Shen-Xu-Zhang:AiCM:14} and Theorem \ref{thm:itr_con}.

To close this subsection, we remark that both algorithms \eqref{eq:al1} and \eqref{eq:al2}  do not take advantage of the separability of function $f$ and vector $x$. More precisely, the information of $x_j^{k+1}$ for $j=1,\dots, i-1$ is not used when we update $x_i^{k+1}$. We dedicate the next two subsections to developing new algorithms which make use of the block-wise Gauss-Seidel technique to update blocks $x_1, \dots, x_2, y$.

\subsection{Convergent Implicit Two-step Proximity Algorithms}\label{subsec:implicitAlg}
In this subsection, we propose a two-step implicit fixed-point proximity algorithm from iterative scheme \eqref{eq:general-ite}.
We begin with constructing the set of matrices $\mathcal{M}:=\{M_0, M_1, M_2\}$ by setting
 \begin{equation}\label{eq:MatrixMPADMMm0}
 M_0:=\begin{bmatrix}
 \frac{\beta}{\alpha_1}I&-\beta A_1^\top A_2&-\beta A_1^\top A_3&\dots&-\beta A_1^\top A_s&0\\
 0&\frac{\beta}{\alpha_2}I&-\beta A_2^\top A_3&\dots&-\beta A_2^\top A_s&0\\
 0&0&\frac{\beta}{\alpha_3}I&\dots&-\beta A_3^\top A_s&0\\
 \dots&\dots&\dots&\ddots&\dots&\dots\\
 0&0&0&\dots&\frac{\beta}{\alpha_s}I&0\\
 0&0&0&\dots&0&\frac{1}{\beta}I
 \end{bmatrix},
 \end{equation}
   \begin{equation}\label{eq:MatrixMPADMMm1}
 M_1:=\begin{bmatrix}
 \frac{\beta}{\alpha_1}I&-2\beta A_1^\top A_2&-2\beta A_1^\top A_3&\dots&-2\beta A_1^\top A_s&0\\
 0&\frac{\beta}{\alpha_2}I&-2\beta A_2^\top A_3&\dots&-2\beta A_2^\top A_s&0\\
 0&0&\frac{\beta}{\alpha_3}I&\dots&-2\beta A_3^\top A_s&0\\
 \dots&\dots&\dots&\ddots&\dots&\dots\\
 0&0&0&\dots&\frac{\beta}{\alpha_s}I&0\\
 0&0&0&\dots&0&\frac{1}{\beta}I
 \end{bmatrix},
 \end{equation}
 \begin{equation}\label{eq:MatrixMPADMMm2}
 M_2:=\begin{bmatrix}
 0&\beta A_1^\top A_2&\beta A_1^\top A_3&\dots&\beta A_1^\top A_s&0\\
 0&0&\beta A_2^\top A_3&\dots&\beta A_2^\top A_s&0\\
 0&0&0&\dots&\beta A_3^\top A_s&0\\
 \dots&\dots&\dots&\ddots&\dots&\dots\\
 0&0&0&\dots&0&0\\
 0&0&0&\dots&0&0
 \end{bmatrix}.
 \end{equation}
With this choice of matrices $M_0, M_1, M_2$, noting that $\mathrm{prox}_{\beta\iota_{C}^*}(y)=y-\beta b$,
iterative scheme \eqref{eq:general-ite} leads to
$$
y^{k+1}=y^k+\beta(\sum_{i=1}^s A_ix_i^{k+1}-b).
$$
We then replace $y^{k+1}$ by $y^k+\beta(\sum_{i=1}^s A_ix_i^{k+1}-b)$ as we update $x_i^{k+1}$ for $i\in\mathbb{N}_s$ in  iterative scheme \eqref{eq:general-ite}, we obtain that
\begin{equation}\label{eq:algPADMM}
 \begin{cases}
   x_j^{k+1}=\mathrm{prox}_{\frac{\alpha_j}{\beta}f_j}(x_j^k-\alpha_jA_j^\top(\sum_{i=1}^j A_ix_i^{k+1}+\sum_{i=j+1}^sA_i(2x_i^k-x_i^{k-1})-b)- \frac{\alpha_j}{\beta}A_j^\top y^k), j\in\mathbb{N}_s,\\
    y^{k+1}=y^k+\beta(\sum_{i=1}^sA_ix_i^{k+1}-b).
 \end{cases}
 \end{equation}


We point out the connections of the proposed algorithm \eqref{eq:algPADMM} with the proximal ADMM (PADMM). To this end, we introduce the augmented Lagrangian function for \eqref{model}
\begin{equation}\label{def:l}
\mathcal{L}(x_1, \dots, x_s, y):=\sum_{i=1}^s f_i(x_i)+\frac{\beta}{2}\|\sum_{i=1}^s A_ix_i-b\|_2^2+\langle y, \sum_{i=1}^s{A_ix_i}-b\rangle.
\end{equation}
The PADMM for \eqref{model} reads as
\begin{equation}\label{eq:PADMM}
\begin{cases}
x_j^{k+1}=\arg\min\{\mathcal{L}(x_1^{k+1},\dots,x_{j-1}^{k+1}, x_j, x_{j+1}^k,\dots,  x_s^k, y^k)+\frac{\beta}{2\alpha_j}\|x_j-x_j^k\|_{2}^2: x_j\in\mathbb{R}^{n_j}\}, j\in\mathbb{N}_s,\\
y^{k+1}=y^k+\beta(\sum_{i=1}^s{A_ix_i^{k+1}-b}).\\
\end{cases}
\end{equation}
On the other hand, by the definition of proximity operator \eqref{def:prox}, the proposed algorithm \eqref{eq:algPADMM} can be equivalently rewritten as
\begin{equation}\label{eq:algSpecificPADMM}
\begin{cases}
x_j^{k+1}=\arg\min\{\mathcal{L}(x_1^{k+1},\dots,x_{j-1}^{k+1}, x_j,\tilde x_{j+1}^k,\dots, \tilde x_s^k, y^k)+\frac{\beta}{2\alpha_j}\|x_j-x_j^k\|_{2}^2: x_j\in\mathbb{R}^{n_j}\}, j\in\mathbb{N}_s,\\
y^{k+1}=y^k+\beta(\sum_{i=1}^s{A_ix_i^{k+1}-b}),\\
\tilde x_j^{k+1}=2x_j^{k+1}-x_j^k, j\in\mathbb{N}_s.
\end{cases}
\end{equation}
We can observe that our proposed algorithm \eqref{eq:algPADMM} reduces to the PADMM if we set $\tilde x_j^{k+1}=x_j^{k+1}$  for $j\in\mathbb{N}_s$ in \eqref{eq:algSpecificPADMM}. As shown in \cite{Chen-He-Ye-Yuan:MathProg2014},  convergence of ADMM directly applied to problem \eqref{model} with $s\ge 3$ is not guaranteed. Also, it was shown in \cite{Li-Sun-Toh:proxADMM2015} that PADMM may not converge unless extra assumptions on $f_i$ for $i\in\mathbb{N}_s$  are added. However,  algorithm \eqref{eq:algSpecificPADMM} is ensured to converge without  extra assumptions on $f_i$ for $i\in\mathbb{N}_s$.
We next establish the convergence result of  algorithm \eqref{eq:algPADMM}.
\begin{proposition}\label{prop:padmm}
Let $M_0, M_1, M_2$ be defined as \eqref{eq:MatrixMPADMMm0}, \eqref{eq:MatrixMPADMMm1} and \eqref{eq:MatrixMPADMMm2}. Let $\widetilde M_2:=\frac{1}{\beta}M_2$. If
\begin{equation}\label{cond:padmm}
0<\alpha_i<\frac{1}{2\|\widetilde M_2\|_2},~\mathrm{for~} i\in\mathbb{N}_s ~~\mathrm{and}~~ \beta>0,
\end{equation}
then the set $\{M_0, M_1, M_2\}$ satisfies Condition-M.
\end{proposition}
\begin{proof}
Clearly, we see that $M_0=M_1+M_2$, that is, Item (i) of Condition-M holds. Define $H:=M_0+M_2$. Then $H=\mathrm{diag}(\frac{\beta}{\alpha_1}\mathbf{1}_{n_1},\dots,\frac{\beta}{\alpha_s}\mathbf{1}_{n_s}, 1/\beta \mathbf{1}_m)$ is diagonal and symmetric. Item (ii) of Condition-M is trivial due to $\alpha_i>0$ for $i\in\mathbb{N}_s$ and $\beta>0$.
We then prove the validity of Item (iii) of Condition-M. Since the last $m$ columns and  rows of $M_2$ are all zeros, we have that
$$
\|H^{-1/2}M_2H^{-1/2}\|_2=\|\widetilde H^{-1/2} M_2 \widetilde H^{-1/2}\|_2,
$$
where $\widetilde H:=\mathrm{diag}(\frac{\beta}{\alpha_1}\mathbf{1}_{n_1},\dots,\frac{\beta}{\alpha_s}\mathbf{1}_{n_s}, \mathbf{0}_m)$.  By using  hypothesis \eqref{cond:padmm}, we find that
$$
\|\widetilde H^{-1/2} M_2\widetilde H^{-1/2}\|_2\le \max\{\alpha_i: i\in\mathbb{N}_s\}\|\widetilde M_2\|_2<\frac{1}{2},
$$
which leads to Item (iii) of Condition-M.
\end{proof}

The convergence results of  algorithm \eqref{eq:algPADMM} is presented below.
\begin{theorem}\label{thm:SpecificPADMMconverge}
  Let $f_i \in \Gamma_0(\mathbb{R}^{n_i})$, $A_i$ an $m \times n_i$ matrix, 
   $\alpha_i>0$ for $i\in\mathbb{N}_s$ and $\beta>0$.  Let the sequence $\{(x^k, y^k):k\in\mathbb{N}\}$ be generated from the algorithm \eqref{eq:algSpecificPADMM} for any $(x^0, y^0), (x^1, y^1) \in\mathbb{R}^n\times\mathbb{R}^m$. Let $M_2$ be defined as \eqref{eq:MatrixMPADMMm2} and $\widetilde M_2 = \frac{1}{\beta} M_2$. If the condition \eqref{cond:padmm} is satisfied, then the sequence $\{x^k:k\in\mathbb{N}\}$ converges to a solution of problem \eqref{model}.
 \end{theorem}
 \begin{proof}
 By Proposition \ref{prop:padmm} and Theorem \ref{thm:itr_con}, it suffices to prove $T_\mathcal{M}$ is well-defined when $M_0, M_1, M_2$ are defined by \eqref{eq:MatrixMPADMMm0}, \eqref{eq:MatrixMPADMMm1} and \eqref{eq:MatrixMPADMMm2}. In this case, if $w=T_\mathcal{M}(u, v)$, where $w:=(w_1, \dots, w_s, w_y), u:=(u_1, \dots, u_s, u_y), v:=(v_1,\dots, v_s, v_y) \in\mathbb{R}^{n_1}\times\dots\times\mathbb{R}^{n_s}\times \mathbb{R}^m $, then $w_y=u_y+\sum_{j=1}^s A_j w_j$ and each $w_i$ for $i\in\mathbb{N}_s$ can be calculated by
$$w_i=\arg\min\{f_i(x_i)+\frac{\beta}{2}\|\sum_{j=1}^{i-1}A_jw_j+A_ix_i+\sum_{j=i+1}^{s}A_j(2u_j-v_j)-b\|_2^2+\langle u_y, A_i x_i\rangle +\frac{\beta}{2\alpha_i}\|x_i-u_i\|_{2}^2: x_i\in\mathbb{R}^{n_i}\}.$$
Since the objective function of the above optimization problem is strongly convex, $T_\mathcal{M}$ is well-defined.
\end{proof}

 To end  this subsection, we point out that compared with algorithms \eqref{eq:al1} and \eqref{eq:al2}, algorithm  \eqref{eq:algPADMM} takes advantage of the separable structure of variable $x$ and applies the block-wise Gauss-Seidel technique  to blocks $x_1,\dots, x_s, y$. We also note that solving the subproblems involved in \eqref{eq:algPADMM} may require inner iterations. In practice, it will  affect the computational efficiency of the algorithm \eqref{eq:algPADMM}. In the next subsection, we develop  an explicit two-step algorithm. As long as the proximity operators of $f_i$ for $i\in\mathbb{N}_s$ have  closed form solutions, the algorithm can be implemented efficiently.

 \subsection{Convergent Explicit Two-step Proximity Algorithms}
 In this subsection, we propose a class of explicit algorithms, which apply the block-wise Gauss-Seidel technique to blocks $x_1, \dots, x_s, y$.

 We begin with specifying the set of matrices $\mathcal{M}$. We set
 \begin{equation}\label{eq:MatrixMlADMMm0}
 M_0:=\begin{bmatrix}
 \frac{\beta}{\alpha_1}I-\beta A_1^\top A_1&-\beta A_1^\top A_2&-\beta A_1^\top A_3&\dots&-\beta A_1^\top A_s&0\\
 0&\frac{\beta}{\alpha_2}I-\beta A_2^\top A_2&-\beta A_2^\top A_3&\dots&-\beta A_2^\top A_s&0\\
 0&0&\frac{\beta}{\alpha_3}I-\beta A_3^\top A_3&\dots&-\beta A_3^\top A_s&0\\
 \dots&\dots&\dots&\ddots&\dots&\dots\\
 0&0&0&\dots&\frac{\beta}{\alpha_s}I-\beta A_s^\top A_s&0\\
 0&0&0&\dots&0&\frac{1}{\beta}I
 \end{bmatrix},
 \end{equation}
   \begin{equation}\label{eq:MatrixMlADMMm1}
 M_1:=\begin{bmatrix}
 \frac{\beta}{\alpha_1}I-\beta A_1^\top A_1&-2\beta A_1^\top A_2&-2\beta A_1^\top A_3&\dots&-2\beta A_1^\top A_s&0\\
 0&\frac{\beta}{\alpha_2}I-\beta A_2^\top A_2&-2\beta A_2^\top A_3&\dots&-2\beta A_2^\top A_s&0\\
 0&0&\frac{\beta}{\alpha_3}I-\beta A_3^\top A_3&\dots&-2\beta A_3^\top A_s&0\\
 \dots&\dots&\dots&\ddots&\dots&\dots\\
 0&0&0&\dots&\frac{\beta}{\alpha_s}I-\beta A_s^\top A_s&0\\
 0&0&0&\dots&0&\frac{1}{\beta}I
 \end{bmatrix},
 \end{equation}
and let $M_2$ be defined as in \eqref{eq:MatrixMPADMMm2}.
We can obtain an implicit algorithm by directly substituting  \eqref{eq:MatrixMlADMMm0}, \eqref{eq:MatrixMlADMMm1} and \eqref{eq:MatrixMPADMMm2} into the iterative scheme \eqref{eq:general-ite}. As the same as the algorithm \eqref{eq:algPADMM}, it implies
$$
 y^{k+1}=y^k+\beta(\sum_{i=1}^sA_ix_i^{k+1}-b).
$$
As in subsection  \ref{subsec:implicitAlg}, we  replace $y^{k+1}$ by $y^k+\beta(\sum_{i=1}^sA_ix_i^{k+1}-b)$ when we update  $x_i^{k+1}$ for $i\in\mathbb{N}_s$. This leads to the following explicit algorithm
 \begin{equation}\label{eq:algLADMM}
 (\mathrm{2SFPPA})\begin{cases}
   x_j^{k+1} = &\mathrm{prox}_{\frac{\alpha_j}{\beta}f_j}(x_j^k-\alpha_j A_j^\top(\sum_{i=1}^{j-1}A_ix_i^{k+1}+A_jx_j^k + \\
    & \sum_{i=j+1}^s A_i(2x_i^k-x_i^{k-1})-b)-\frac{\alpha_j}{\beta}A_j^\top y^k), ~~ j\in\mathbb{N}_s,\\
    y^{k+1} =& y^k+\beta(\sum_{i=1}^s A_ix_i^{k+1}-b).
  \end{cases}
 \end{equation}

We point out here the relationship between the proposed algorithm \eqref{eq:algLADMM} and the LADMM. To this end, we first review the exact extension of LADMM to problem \eqref{model}. For $j\in\mathbb{N}_s$, let $J_j:\mathbb{R}^{n_1}\times\dots\times\mathbb{R}^{n_s}\times\mathbb{R}^m\rightarrow\mathbb{R}^{n_j}$ defined, at $(x_1, x_2,\dots, x_s, y)$, as $J_j(x_1,\dots, x_s, y):=\beta  (A_j^\top\sum_{i=1}^s A_ix_i-b)+A_j^\top y$.
The direct extension of LADMM to the multi-block problem is as follows
  \begin{equation}\label{LADMM}
 \begin{cases}
  x_j^{k+1}= & \arg\min\{f_j(x_j)+\langle  J_j(x_1^{k+1}, \dots, x_{j-1}^{k+1}, x_j^k,  \dots,  x_s^k, y^k), x_j-x_j^k\rangle \\
  & +\frac{\beta}{2\alpha_j}\|x_j-x_j^k\|_{2}^2: x_j\in\mathbb{R}^{n_j}\},~~j\in\mathbb{N}_s,\\
    y^{k+1}= & y^k+\beta(\sum_{i=1}^sA_ix_i^{k+1}-b).\\
 \end{cases}
 \end{equation}
 Using the above notations and the definition of proximity operators \eqref{def:prox}, the algorithm \eqref{eq:algLADMM} can be rewritten in its equivalent form
  \begin{equation}\label{MLADMM}
 \begin{cases}

  x_j^{k+1}= & \arg\min\{f_j(x_j)+\langle  J_j(x_1^{k+1}, \dots, x_{j-1}^{k+1}, x_j^k, \tilde x_{j+1}^k, \dots, \tilde x_s^k, y^k), x_j-x_j^k\rangle \\
  & +\frac{\beta}{2\alpha_j}\|x_j-x_j^k\|_{2}^2: x_j\in\mathbb{R}^{n_j}\},j\in\mathbb{N}_s,\\
    y^{k+1}= &y^k+\beta(\sum_{i=1}^sA_ix_i^{k+1}-b),\\
  \tilde x^{k+1}= & 2x^{k+1}-x^k.
 \end{cases}
 \end{equation}
Obviously,  our proposed algorithm \eqref{eq:algLADMM} reduces to the LADMM if we set $\tilde x^{k+1}=x^{k+1}$ in \eqref{MLADMM}.
As mentioned in \cite{He-Yuan:LADMGBS:2013}, the direct extension of LADMM to the multi-block problem \eqref{model} is not necessarily convergent. Nevertheless, the convergence of the proposed algorithm \eqref{eq:algLADMM} is guaranteed. Next we present the convergence results of the algorithm \eqref{eq:algLADMM}.
\begin{proposition}\label{prop:ladmm}
Let $M_0$, $M_1$ and $M_2$ be defined as in \eqref{eq:MatrixMlADMMm0}, \eqref{eq:MatrixMlADMMm1} and \eqref{eq:MatrixMPADMMm2}. Let $\widetilde M_2:=\frac{1}{\beta}M_2. $ If
\begin{equation}\label{cond:ladmm}
0<\alpha_i<\frac{1}{\|A_i\|^2_2+2\|\widetilde M_2\|_2},~\mathrm{for~} i\in\mathbb{N}_s ~~\mathrm{and}~~ \beta>0,
\end{equation}
then the set $\{M_0, M_1, M_2\}$ satisfies Condition-M.
\end{proposition}
\begin{proof}
 It is clear that Item (i) of Condition-M is satisfied. Define $H:=M_0+M_2$. Then $H=\mathrm{diag}(\frac{\beta}{\alpha_1}I-\beta A_1^\top A_1,\dots,\frac{\beta}{\alpha_s}I-\beta A_s^\top A_s, 1/\beta I)$ is symmetric. In light of \eqref{cond:ladmm}, we have that $H\in\mathbb{S}_+^{n+m}$. Hence, Item (ii) of Condition-M holds.
We next show Item (iii) of Condition-M. Similar to the proof of Proposition \ref{prop:padmm},
$$
\|H^{-1/2}M_2H^{-1/2}\|_2=\|\widetilde H^{-1/2} M_2 \widetilde H^{-1/2}\|_2,
$$
where  $\widetilde H:=\mathrm{diag}(\frac{\beta}{\alpha_1}I-\beta A_1^\top A_1,\dots,\frac{\beta}{\alpha_s}I-\beta A_s^\top A_s, 0)$. Hypothesis \eqref{cond:ladmm} leads to
$$
\|\widetilde H^{-1/2} M_2\widetilde H^{-1/2}\|_2\le \max\left\{\frac{1}{\frac{1}{\alpha_i}-\|A_i\|^2_2}: i\in\mathbb{N}_s\right\}\|\widetilde M_2\|_2<\frac{1}{2}.
$$
This completes the proof.
\end{proof}

The following theorem regards the convergence of  algorithm \eqref{eq:algLADMM}.
 \begin{theorem}
  Let $f_i \in \Gamma_0(\mathbb{R}^{n_i})$, $A_i$ an $m \times n_i$ matrix, 
  $\alpha_i>0$ for $i\in\mathbb{N}_s$ and $\beta>0$.  Let the sequence $\{(x^k, y^k):k\in\mathbb{N}\}$ be generated from algorithm \eqref{eq:algLADMM} for any $(x^0, y^0), (x^1, y^1) \in\mathbb{R}^n\times\mathbb{R}^m$. Let $M_2$ be defined as \eqref{eq:MatrixMPADMMm2} and $\widetilde M_2 = \frac{1}{\beta} M_2$. If  condition \eqref{cond:ladmm} is satisfied, then the sequence $\{x^k:k\in\mathbb{N}\}$ converges to a solution of problem \eqref{model}.
 \end{theorem}
 \begin{proof}
By Theorem \ref{thm:itr_con} and Proposition \ref{prop:ladmm}, we only need to prove $T_\mathcal{M}$ is well-defined. In this case, from   algorithm \eqref{eq:algLADMM}, it is obvious that $T_\mathcal{M}$ can be computed explicitly. Therefore $T_\mathcal{M}$ is well-defined.
 \end{proof}

To close this subsection, we remark  that when the proximity operators of $f_i$ for $i\in\mathbb{N}_s$ have closed form solutions, the two-step algorithm \eqref{eq:algLADMM} may be more efficient than the two-step algorithm \eqref{eq:algPADMM}. This is because  the two-step algorithm \eqref{eq:algPADMM} may require inner iterations to solve the subproblems involved, while each step of  algorithm \eqref{eq:algLADMM} can be implemented efficiently by making use of the closed form.

\subsection{Variants of algorithms \eqref{eq:algPADMM} and \eqref{eq:algLADMM}}
There is a wide variety of the choices of $\{M_0, M_1, M_2\}$ satisfying condition-M, including those of  algorithms \eqref{eq:algPADMM} and \eqref{eq:algLADMM}. In this subsection, we present  other choices of  $\{M_0, M_1, M_2\}$ satisfying condition-M. With these choices the two step iterative scheme \eqref{eq:general-ite} reduces to a class of new algorithms, which can be viewed as variants of  algorithms \eqref{eq:algPADMM} and \eqref{eq:algLADMM}.

{\bf{Modifications of diagonal blocks:}}  The diagonal blocks of $M_1$ and $M_2$ can be chosen in  other ways. We only present two examples in the following. For instance, the diagonal entries of $M_1$ in \eqref{eq:MatrixMPADMMm1} can be chosen as $((\theta+1)\frac{\beta}{\alpha_1}{\bf{1}}_{n_1},\dots, (\theta+1)\frac{\beta}{\alpha_s}{\bf{1}}_{n_s}, \frac{1}{\beta}{\bf{1}}_{m})$  with $\theta\in[0, 1)$ and correspondingly, the  diagonal entries of $M_2$ in \eqref{eq:MatrixMPADMMm2} should be $(-\theta \frac{\beta}{\alpha_1}{\bf{1}}_{n_1},\dots, -\theta \frac{\beta}{\alpha_s}{\bf{1}}_{n_s}, {\bf{0}}_m)$. With such a choice of $\{M_0, M_1, M_2\}$, iterative scheme \eqref{eq:general-ite} reduces to a variant of algorithm \eqref{eq:algPADMM}
    $$
    \begin{cases} x_j^{k+1}=&\mathrm{prox}_{\frac{\alpha_j}{\beta}f_j}(x_j^k+\theta(x_j^k-x_j^{k-1})-\alpha_jA_j^\top(\sum_{i=1}^j A_ix_i^{k+1}\\
    &+\sum_{i=j+1}^sA_i(2x_i^k-x_i^{k-1})-b)- \frac{\alpha_j}{\beta}A_j^\top y^k), j\in\mathbb{N}_s,\\
    y^{k+1}=&y^k+\beta(\sum_{i=1}^s A_ix_i^{k+1}-b).
    \end{cases}
    $$  As a second example, the diagonal blocks of $M_1$ in \eqref{eq:MatrixMlADMMm1} can be chosen as $(\frac{\beta}{\alpha_1}I-2\beta A_1^\top A_1,\dots, \frac{\beta}{\alpha_s}I-2\beta A_s^\top A_s, \frac{1}{\beta}I)$. Accordingly,  the  diagonal blocks of $M_2$ in \eqref{eq:MatrixMPADMMm2} should be $(\beta A_1^\top A_1,\dots, \beta A_s^\top A_s, {\bf{0}})$ to make $M_0=M_1+M_2$. These matrices leads to a variant of algorithm \eqref{eq:algLADMM}
 $$
    \begin{cases} x_j^{k+1} = & \mathrm{prox}_{\frac{\alpha_j}{\beta}f_j}(x_j^k-\alpha_j A_j^\top(\sum_{i=1}^{j-1}A_ix_i^{k+1}\\
    &+\sum_{i=j}^s A_i(2x_i^k-x_i^{k-1})-b)-\frac{\alpha_j}{\beta}A_j^\top y^k), ~~ j\in\mathbb{N}_s,\\
    y^{k+1}=&y^k+\beta(\sum_{i=1}^s A_ix_i^{k+1}-b).
    \end{cases}
    $$

{\bf Modifications of nondiagonal blocks:}
  We change the $(i, j)$-th block of  $M_0$ (defined by \eqref{eq:MatrixMPADMMm0} or \eqref{eq:MatrixMlADMMm0}) for $i>j$  from $0$ to $\theta\beta A_i^\top A_j$ and keep other blocks of $M_0$  unchanged. In order to make $M_0+M_2$ symmetric, the matrix $M_2$ should be chosen as $\theta+1$ multiplying the original matrix $M_2$ defined in \eqref{eq:MatrixMPADMMm2}. Accordingly,   the matrix $M_1$ can be determined by $M_1=M_0-M_2$. Then we can derive the following two algorithms from iterative scheme \eqref{eq:general-ite}
  \begin{equation}
 \begin{cases}
   x_j^{k+1}=&\mathrm{prox}_{\frac{\alpha_j}{\beta}f_j}(x_j^k-\alpha_jA_j^\top(\sum_{i=1}^{j-1} A_i(x_i^{k+1}+\theta(x_i^{k+1}-x_i^k))+A_jx_j^{k+1}\\
   &+\sum_{i=j+1}^sA_i((2+\theta)x_i^k-(\theta+1)x_i^{k-1})-b)- \frac{\alpha_j}{\beta}A_j^\top y^k), ~~j\in\mathbb{N}_s,\\
    y^{k+1}=&y^k+\beta(\sum_{i=1}^sA_ix_i^{k+1}-b),
 \end{cases}
 \end{equation}
 \begin{equation}
 \begin{cases}
   x_j^{k+1} = &\mathrm{prox}_{\frac{\alpha_j}{\beta}f_j}(x_j^k-\alpha_j A_j^\top(\sum_{i=1}^{j-1}A_i(x_i^{k+1}+\theta(x_i^{k+1}-x_i^k))+A_jx_j^k \\
    & +\sum_{i=j+1}^s A_i((2+\theta)x_i^k-\theta x_i^{k-1})-b)-\frac{\alpha_j}{\beta}A_j^\top y^k), ~~ j\in\mathbb{N}_s,\\
    y^{k+1} =& y^k+\beta(\sum_{i=1}^s A_ix_i^{k+1}-b).
  \end{cases}
 \end{equation}

{\bf Hybrids of both algorithms:} Both algorithms \eqref{eq:algPADMM} and \eqref{eq:algLADMM} share the same matrix $M_2$. Matrices $M_0$ for algorithms \eqref{eq:algPADMM} and  \eqref{eq:algLADMM} are almost the same except the diagonal blocks.  Let $S_1\subseteq\mathbb{N}_s$ and $S_2=\mathbb{N}_s\backslash S_1$. Suppose the subproblems of \eqref{eq:algSpecificPADMM} for $x_i^{k+1}$, $i\in S_1$  can be solved efficiently.  We also assume inner iterations are required to solve the subproblems of \eqref{eq:algSpecificPADMM} for $x_i^{k+1}$, $i\in S_2$.
    We set the $i$-th diagonal block of $M_0$ to be $\frac{\beta}{\alpha_i}I$ for $i\in S_1$ and to be $\frac{\beta}{\alpha_i}I-\beta A_i^\top A_i$ for $i\in S_2$. The nondiagonal blocks of $M_0$ are chosen to be the same as in \eqref{eq:MatrixMPADMMm0} and \eqref{eq:MatrixMlADMMm0}. We further choose the matrix $M_2$ as in \eqref{eq:MatrixMPADMMm2}. Accordingly, the matrix $M_1$ is determined by $M_1=M_0-M_2$. Then we obtain the following hybrid algorithm
     $$
     \begin{cases}
     x_j^{k+1}=&\mathrm{prox}_{\frac{\alpha_j}{\beta}f_j}(x_j^k-\alpha_jA_j^\top(\sum_{i=1}^j A_ix_i^{k+1}\\
     &+\sum_{i=j+1}^sA_i(2x_i^k-x_i^{k-1})-b)- \frac{\alpha_j}{\beta}A_j^\top y^k),\mathrm{~if~}  j\in S_1,\\
     x_j^{k+1} = &\mathrm{prox}_{\frac{\alpha_j}{\beta}f_j}(x_j^k-\alpha_j A_j^\top(\sum_{i=1}^{j-1}A_ix_i^{k+1}+A_jx_j^k + \\
    & \sum_{i=j+1}^s A_i(2x_i^k-x_i^{k-1})-b)-\frac{\alpha_j}{\beta}A_j^\top y^k), \mathrm{~if~} j\in S_2,\\
    y^{k+1}=&y^k+\beta(\sum_{i=1}^s A_ix_i^{k+1}-b).
     \end{cases}
     $$

We point out here that  convergence of the above algorithms is guaranteed. One can obtain the convergence results by verifying that the corresponding set of matrices $\mathcal{M}:=\{M_0, M_1, M_2\}$ satisfies Condition-M and $T_\mathcal{M}$ is well-defined. We omit the details here since the proofs are  similar to those of algorithms \eqref{eq:algPADMM} and \eqref{eq:algLADMM}.

\section{Numerical Experiments}\label{sec:exp}
In this section, we demonstrate the efficiency of the proposed two-step fixed-point proximity algorithms by applying 2SFPPA to the sparse Magnetic Resonance Imaging (MRI) reconstruction problem \cite{Lustig-Donoho-Pauly:MRM:07}. We shall compare the performances of the proposed 2SFPPA with those of other LADMM-type algorithms.
\subsection{Sparse MRI problem}
For convenience of exposition, we assume that an image considered has a size of $d_1
\times d_2$. The image is treated as a vector in $\mathbb{R}^{d_1 d_2}$ in
such a way its $(i,j)$-th pixel  corresponds to the
$(i+(j-1)d_2)$-th component of the vector in $\mathbb{R}^{d_1 d_2}$.
We
set $d:= d_1 d_2$. Let $K\in\mathbb{R}^{p \times d}$ $(p < d)$ be a partial Fourier transform matrix and $b\in\mathbb{R}^{p}$ represent the observed data. Then the general form of the sparse MRI reconstruction model can be written as
$$
\min\{F(u):u\in\mathbb{R}^d, Ku = b\},
$$
where $F(\cdot):\mathbb{R}^d\rightarrow \mathbb{R}$ is a sparse-promoting function. It is well-known that superior image reconstruction can be obtained when $F(\cdot)$ is chosen to be the hybrid of total variation and the $\ell_1$-norm of the Haar wavelet transform. Denote the Haar wavelet transform matrix by $W\in\mathbb{R}^{q \times d}$ and define the $q\times q$ diagonal matrix $\Lambda:=\mathrm{diag}(\lambda_1,\dots,\lambda_{q})$ with $\lambda_i\geq 0, i\in\mathbb{N}_q$.
We turn to considering the following specific sparse MRI problem
\begin{equation}\label{prob:smri1}
\min\{\mu\|u\|_{\mathrm{TV}} + \|\Lambda Wu\|_1:u\in\mathbb{R}^d, Ku = b\},
\end{equation}
where $\mu>0$ trades the total variation with sparsity of the wavelet coefficients $Wu$.

In order to apply the proposed algorithms, we need to reformulate  problem \eqref{prob:smri1}. First, we rewrite $\|\cdot\|_{\mathrm{TV}}$ to a function composed with a linear mapping. To this end,  we recall the $r \times r$ difference matrix $D_r$ by
\begin{equation}\label{eq:d}
D_r:=\begin{bmatrix}
 1&&&-1 \\
 -1 &1 \\
 \hfill &\ddots &\ddots \\
 \hfill &\hfill &{-1} &{1} \\
\end{bmatrix}.
\end{equation}
Through the matrix Kronecker product $\otimes$, we define the
$2d \times d$ matrix $ B$ by
\begin{equation}\label{def:B}
B:=\begin{bmatrix}
I_{d_2}\otimes D_{d_1}\\
D_{d_2} \otimes I_{d_1}
\end{bmatrix}.
\end{equation}
Moreover, we define function $\psi:\mathbb{R}^{2d}\rightarrow \mathbb{R}$ at
$y \in \mathbb{R}^{2d}$ as
\begin{equation}\label{def:psi}
\psi(y):=\sum_{i=1}^{d} \left\|[y_i, y_{d+i}]^\top\right\|_2.
\end{equation}
With the definition of  matrix $ B$ \eqref{def:B} and the
convex function $\psi$ \eqref{def:psi}, the (isotropic) total variation of an image $x$ can be represented by
\begin{equation}\label{eq:TV}
\|x\|_{\mathrm{TV}}=\psi(Bx).
\end{equation}
Moreover, we define $\varphi:\mathbb{R}^q\rightarrow \mathbb{R}$ at $y\in\mathbb{R}$ as $\varphi(y):=\|\Lambda y\|_1$.
Then with help of the formula \eqref{eq:TV}, function $\varphi$ and the indicator function $\iota_{\{b\}}$,  problem \eqref{prob:smri1} can be equivalently reformulated as
\begin{equation}\label{prob:smri2}
\min\{\mu\psi(Bu) + \varphi(Wu) + \iota_{\{b\}}(Ku) :u\in\mathbb{R}^d\}.
\end{equation}
Recall the dual problem of \eqref{prob:smri2} has a form of
\begin{equation}\label{prob:dualmri1}
\min\{(\mu\psi)^*(x_1) + \varphi^*(x_2) + \iota_{\{b\}}^*(x_3) : B^\top x_1 + W^\top x_2+ K^\top x_3 = 0, x_1\in\mathbb{R}^{2d}, x_2\in\mathbb{R}^q, x_3\in\mathbb{R}^p\}.
\end{equation}
By the definition of the Fenchel conjugate function, one can easily check that the Fenchel conjugate functions in \eqref{prob:dualmri1} have the form
$$
(\mu\psi)^* = \iota_{S_1}, ~~ \varphi^* = \iota_{S_2}, ~\mathrm{and}~ \iota_{\{b\}}^*(\cdot) = \langle b,\cdot \rangle,
$$
where the sets $S_1\subseteq \mathbb{R}^{2d}$ and $S_2\subseteq \mathbb{R}^{q}$ are defined as
$$
S_1: = \{\|[y_i,y_{d+i}]\|_2\leq \mu, \forall i\in\mathbb{N}_d: y\in\mathbb{R}^{2d}\}
$$
and
$$
S_2:= \{|y_j|\leq \lambda_j, \forall j\in\mathbb{N}_q: y\in\mathbb{R}^q\}.
$$
Therefore, we obtain the following minimization problem
\begin{equation}\label{prob:dualmri2}
\min\{\iota_{S_1}(x_1)+\iota_{S_2}(x_2)+\langle b,x_3 \rangle: B^\top x_1 + W^\top x_2+ K^\top x_3 = 0, x_1\in\mathbb{R}^{2d}, x_2\in\mathbb{R}^q, x_3\in\mathbb{R}^p\}.
\end{equation}
Obviously, problem \eqref{prob:dualmri2} is a special case of the multi-block problem \eqref{model} with the block number $s=3$. Thus we can directly apply 2SFPPA to solving  problem \eqref{prob:dualmri2}. In particular, all the proximity operators of the convex functions involved in \eqref{prob:dualmri2} have closed forms. More precisely, the proximity operators
$
\mathrm{prox}_{\frac{\alpha_1}{\beta}\iota_{S_1}}
$
and
$
\mathrm{prox}_{\frac{\alpha_2}{\beta}\iota_{S_2}}
$
are exactly the projection operator onto the sets $S_1$ and $S_2$ respectively. The proximity operator $\mathrm{prox}_{\frac{\alpha_3}{\beta}\langle \cdot, b\rangle}$ is just the shift operator. We describe the 2SFPPA for the sparse MRI model in Algorithm \ref{alg:smri}.
\begin{algorithm}\caption{(2SFPPA for the sparse MRI)}\label{alg:smri}
 \begin{algorithmic}[1]
   \State Given: observed data $b$ in $\mathbb{R}^{p}$; $\mu>0$, $\Lambda\geq 0$,
   $\alpha_1$,$\alpha_2$,$\alpha_3>0$ and $\beta>0$
    \State Initialization: $x^0_1=K^\top b$, $x^0_{2}=x^{-1}_2=0$, $x^0_{3}=x^{-1}_3=0$, $y_0 = 0$.
   \Repeat
   \State Step 1: $x_1^{k+1} \longleftarrow \mathrm{Proj}_{S_1}(x_1^k - \alpha_1 B (B^\top x^k_1 + W^\top (2x_2^k -x_2^{k-1})+ K^\top (2x^k_3 - x_3^{k-1})+ \frac{1}{\beta}y^k))$, \\
   \State Step 2: $x_2^{k+1} \longleftarrow \mathrm{Proj}_{S_2}(x_2^k - \alpha_2 W (B^\top x^{k+1}_1 + W^\top x^k_2+ K^\top (2x^k_3 - x_3^{k-1})+ \frac{1}{\beta}y^k))$, \\
    \State Step 3: $x_3^{k+1} \longleftarrow x_3^k - \alpha_3 K (B^\top x^{k+1}_1 + W^\top x^{k+1}_2+ K^\top x^k_3+ \frac{1}{\beta}y^k)-\frac{\alpha_3}{\beta}b$, \\
    \State Step 4: $y^{k+1}\longleftarrow y^k + \beta (B^\top x^{k+1}_1 + W^\top x^{k+1}_2 + K^\top x^{k+1}_3).$

   \Until{``convergence''}
   \State Write the output of $-y^{k}$ from the above loop as
   $u^{\infty}$.
 \end{algorithmic}
\end{algorithm}

\subsection{Numerical results}
In this subsection, we shall compare numerical results of the proposed 2SFPPA  with those of the Jacobi-type LADMM (JADMM) \eqref{eq:al2}, the LADMM and LADMM  with Gaussian back substitution (LADMMG) for the sparse MRI problem. All the experiments are conducted in Matlab 7.6 (R2008a) installed on a laptop with Intel Core i5 CPU at 2.5GHz, 8G RAM running Windows 7.

In the experiment, we select the $256\times 256$ ``Shepp-Logan'' phantom as the test image, see Fig.1 (a). The observed data $b$ is obtained by sampling the discrete Fourier transform of the phantom along 17 pseudo-radial lines, as shown in Fig.1 (b). The Haar wavelet transform $W\in\mathbb{R}^{p\times d}$ is chosen to be non-decimated and thus we have that $p = 4d$. We assume that the upper $d\times d$ sub-matrix of $W$ is formed by the low-pass filter while the remaining $3d\times d$ sub-matrix is formed by the high-pass filters. Accordingly, we set the diagonal entries of the diagonal matrix $\Lambda$ as follows
$$
\lambda_i =
\begin{cases}
0, ~~i\in\mathbb{N}_d,\\
\frac{1}{2}, ~~i\in\mathbb{N}_p\backslash \mathbb{N}_d.
\end{cases}
$$
We further take the regularization parameters $\mu = 3$ throughout the test.
We measure the computational efficiency of the compared algorithms by two criteria. One criterion is the relative error between values of the objective function at each iteration and the optimal function value of problem \eqref{prob:smri2}. We remark that the indicator function $\iota_{\{b\}}$ is involved in the objective function and the iterates $u^k= -y^k$ may not always satisfy $K u^k = b$. Therefore, for fair numerical comparisons we compute the following relative error
$$
\epsilon^k_1 := (F(u^k)+\tau \|Ku^k-b\|_2  - F^*)/F^*,
$$
where $\tau>0$ is a penalty parameter and $F^*$ denotes the optimal function value. In practice, we set $\tau = 1000$ and run the LADMM for 5000 iterations to obtain an approximation of $F^*$. The other one is that the relative error between two successive iterates
$$
\epsilon^k_2 := \frac{\|y^{k}-y^{k-1}\|_2}{\|y^{k}\|_2}.
$$
The quality of the reconstructed image is evaluated in terms of the peak signal-to-noise ratio (PSNR) defined by
$$
\mathrm{PSNR}=10\mathrm{log}_{10}\frac{255\sqrt{d}}{\|u^\infty-u^\star\|_2}(\mathrm{dB}),
$$
where $u^\star$ is the original image vector and $u^\infty$ is the recovered image vector.
\begin{figure}\label{fig:image}
\begin{center}
\begin{tabular}{cc}
\scalebox{0.3}{\includegraphics{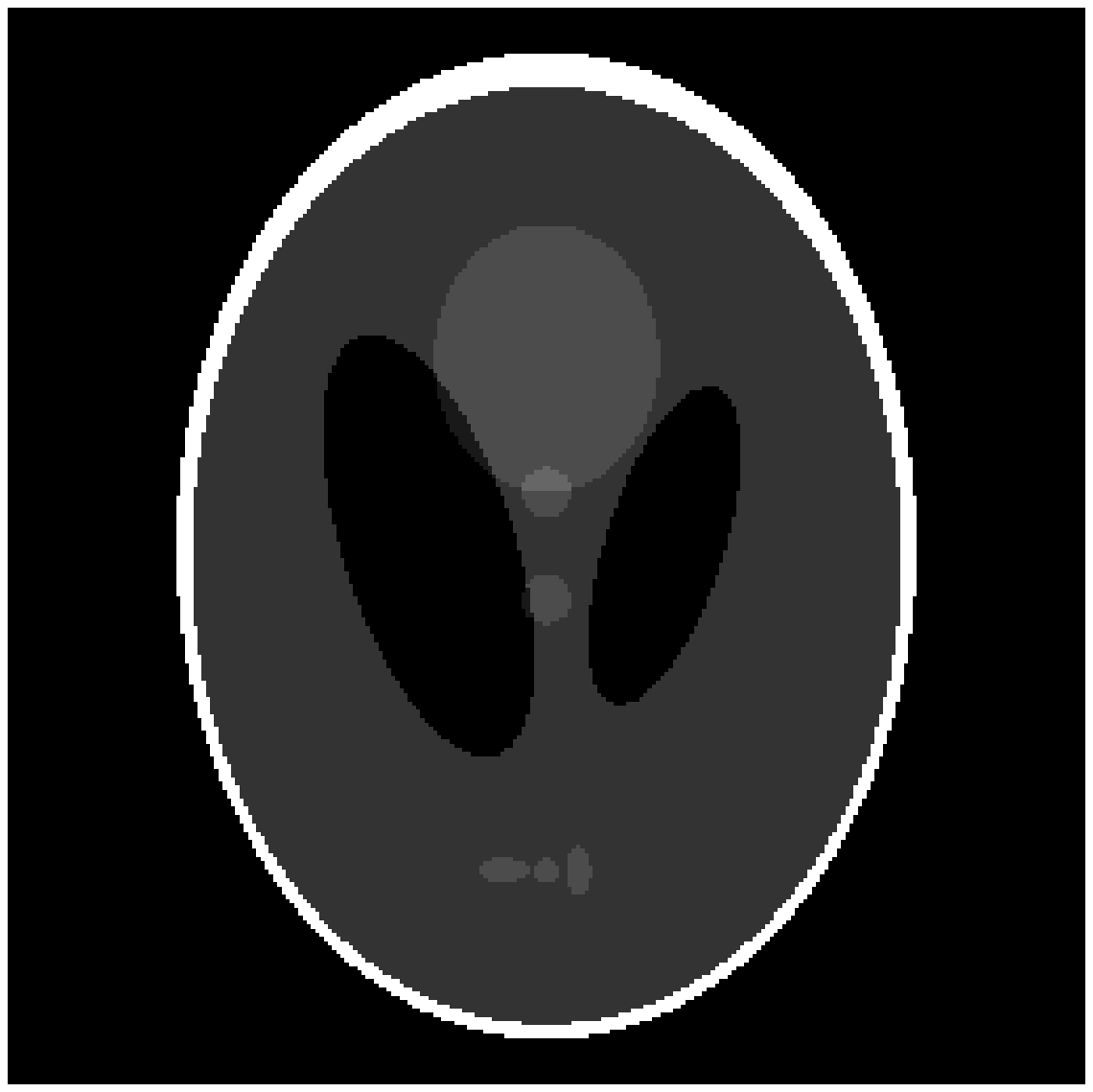}}
&\scalebox{0.3}{\includegraphics{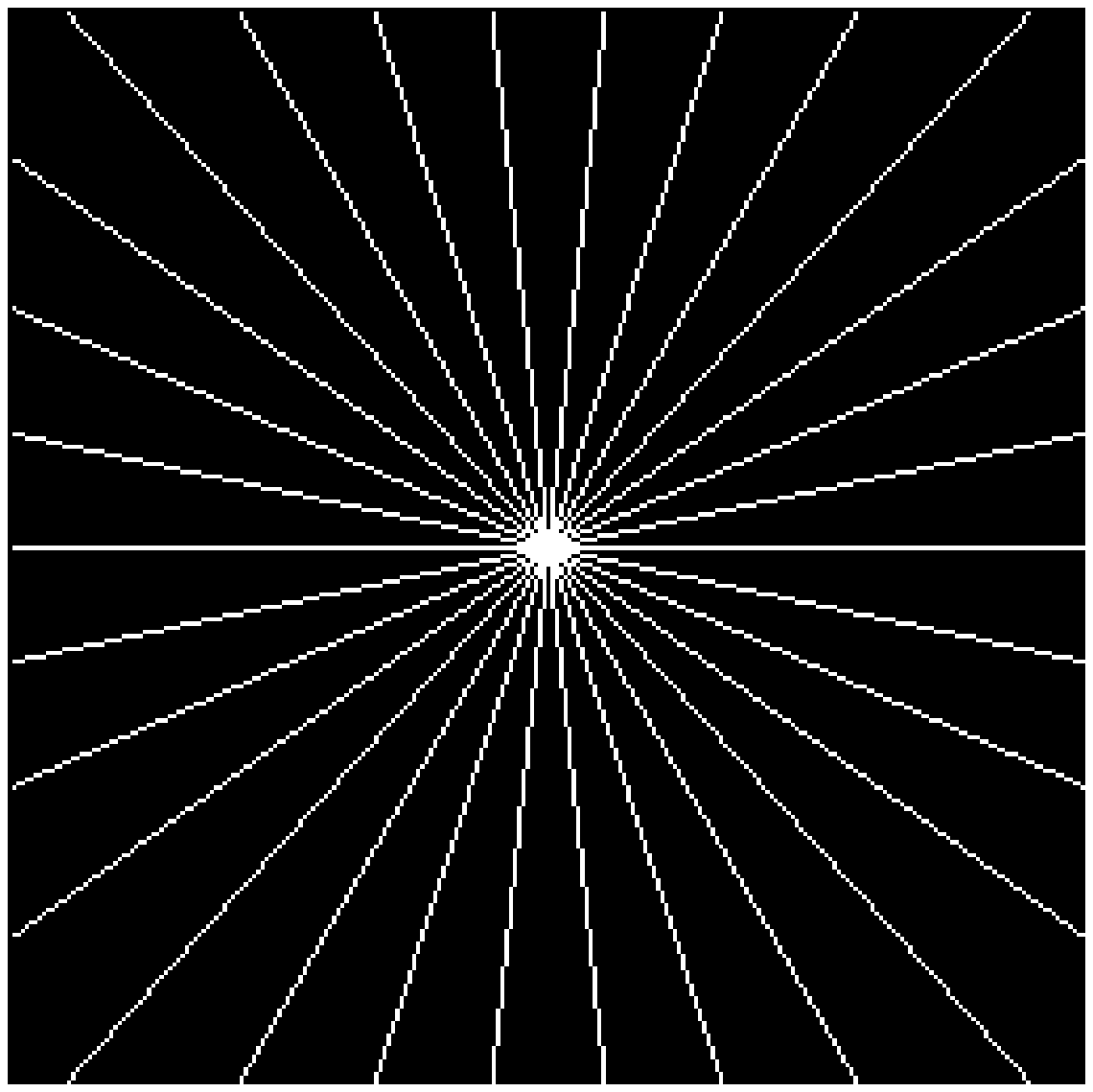}}\\
(a) & (b) \\
\end{tabular}
\end{center}
\caption{\footnotesize{(a) Shepp-Logan phantom (b) Sample pattern.}}
\end{figure}

For the JLADMM, we set
\begin{equation}\label{eq:parameter1}
 \alpha_1 = \alpha_2 = \alpha_3 =\frac{1}{8} ~\mathrm{and}~ \beta=1.
\end{equation}
For the LADMMG, LADMM, and 2SFPPA, we set
 \begin{equation}\label{eq:parameter2}
 \alpha_1 = \frac{1}{8},
 \alpha_2 = \frac{0.999999}{\|W\|^2_2},
 \alpha_3 = \frac{0.999999}{\|K\|^2_2}, ~\mathrm{and}~ \beta = 1.
 \end{equation}
Besides, as suggested in \cite{He-Yuan:LADMGBS:2013}, the parameter $\theta$ involved in LADMMG is set to be $1$. With such choice of  parameters, all the four algorithms achieve their best performance in terms of the convergence speed.

Table 1 and Table 2
 summarize the numbers of iteration, PSNR values and CPU times when the three algorithms achieve the given accuracy. We observe that the proposed 2SFPPA performs slightly better than LADMM and much better than JLADMM and LADMMG in terms of computational time. The LADMMG costs much more CPU time than LADMM and 2SFPPA due to the Gaussian back substitution step which ensures  convergence of the algorithm. The evolution of the objective function values and PSNR values with respect to the CPU time and the number of iterations are shown in Fig.2. The sequence of function values from 2SFPPA decreases faster to the minimum value than that from JLADMM and LADMMG. Similarly, the sequence of PSNR values from 2SFPPA grows faster to the maximum value than that from JLADMM and LADMMG. Overall, we conclude that 2SFPPA performs as efficiently as LADMM and much better than JLADMM and LADMMG.

\begin{table}\label{tab:sum1}
{\caption {\footnotesize Performance comparison for the sparse MRI. For a given error tolerance $\epsilon$, the first column in the bracket represents the first iteration number $k$ such that $\epsilon^k_1<\epsilon$, the second column and the third column in the bracket show the corresponding PSNR and CPU time.}}
\begin{center}

 \begin{tabular}{cccc}
 \hline
 &$\epsilon=10^{-4}$&$\epsilon=10^{-5}$&$\epsilon=10^{-6}$\\ \hline
  $\mathrm{JLADMM}$&$(3410,~~65.68,~~178.79)$&$(-,~~-,~~-)$&$(-,~~-,~~-)$\\
 $\mathrm{LADMMG}$&$(1237,~~63.71,~~119.59)$&$(3667,~~69.43,~~363.04)$&$(-,~~-,~~-)$\\
 $\mathrm{LADMM}$&$(1140,~~63.63,~~64.46)$&$(3452,~~69.29,~~201.41)$&$(4778,~~70.88,~~279.21)$\\
 $\mathrm{2SFPPA}$&$(1026,~~63.61,~~60.12)$&$(3175,~~69.18,~~184.92)$&$(4455,~~70.78,~~259.99)$\\

 \hline
 \end{tabular}

 \end{center}

\end{table}

\begin{table}
{\caption {\footnotesize Performance comparison for the sparse MRI. For a given error tolerance $\epsilon$, the first column in the bracket represents the first iteration number $k$ such that $\epsilon^k_2<\epsilon$, the second column and the third column in the bracket show the corresponding PSNR and CPU time.}}\label{tab:sum2}
\begin{center}

 \begin{tabular}{cccc}
 \hline
 &$\epsilon=5\times 10^{-5}$&$\epsilon=5\times 10^{-6}$&$\epsilon=5\times 10^{-7}$\\ \hline
 $\mathrm{JLADMM}$&$(646,~~55.15,~~33.15)$&$(1388,~~59.97,~~71.79)$&$(3416,~~65.69,~~179.15)$\\
 $\mathrm{LADMMG}$&$(473,~~57.68,~~46.57)$&$(928,~~62.05,~~89.70)$&$(2451,~~67.37,~~242.11)$\\
 $\mathrm{LADMM}$&$(468,~~58.18,~~26.76)$&$(920,~~62.44,~~51.77)$&$(2366,~~67.42,~~137.29)$\\
 $\mathrm{2SFPPA}$&$(438,~~58.26,~~26.22)$&$(909,~~62.98,~~53.50)$&$(2305,~~67.66,~~133.94)$\\

 \hline
 \end{tabular}
 \end{center}
\end{table}

\begin{figure}\label{fig:vsTN}
\begin{center}
\begin{tabular}{cc}
\scalebox{0.4}{\includegraphics{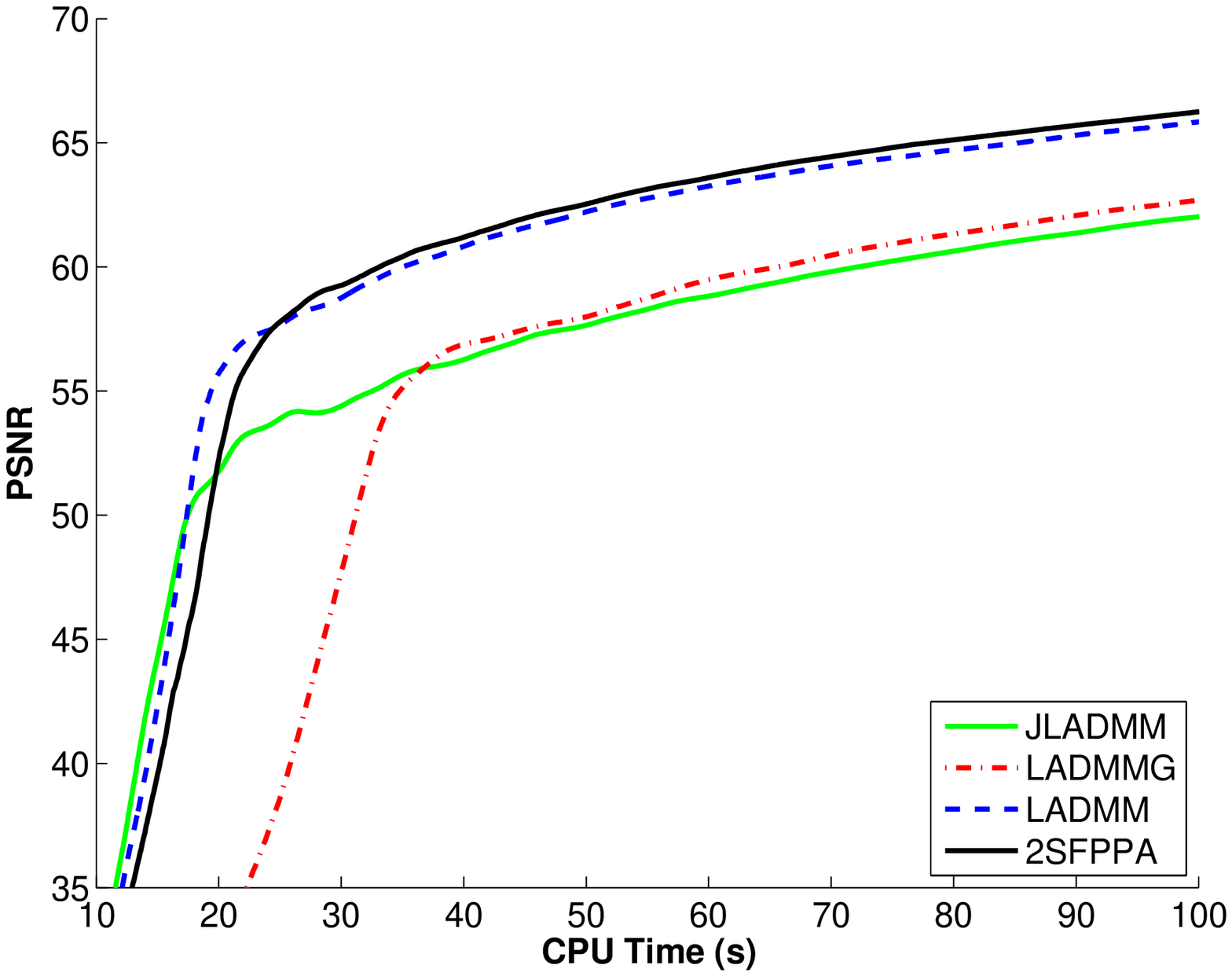}}
&\scalebox{0.4}{\includegraphics{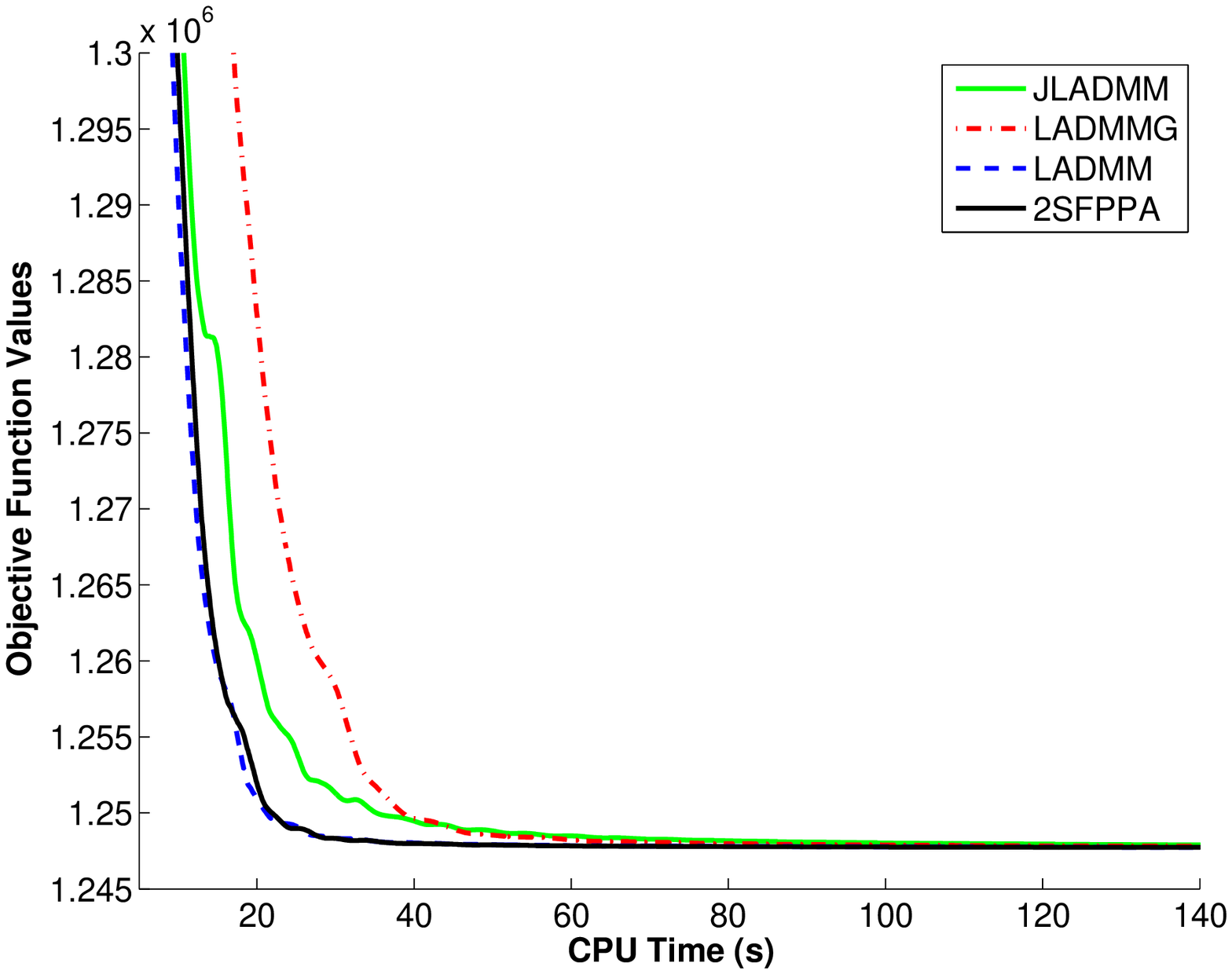}}\\
(a) & (b) \\
\scalebox{0.4}{\includegraphics{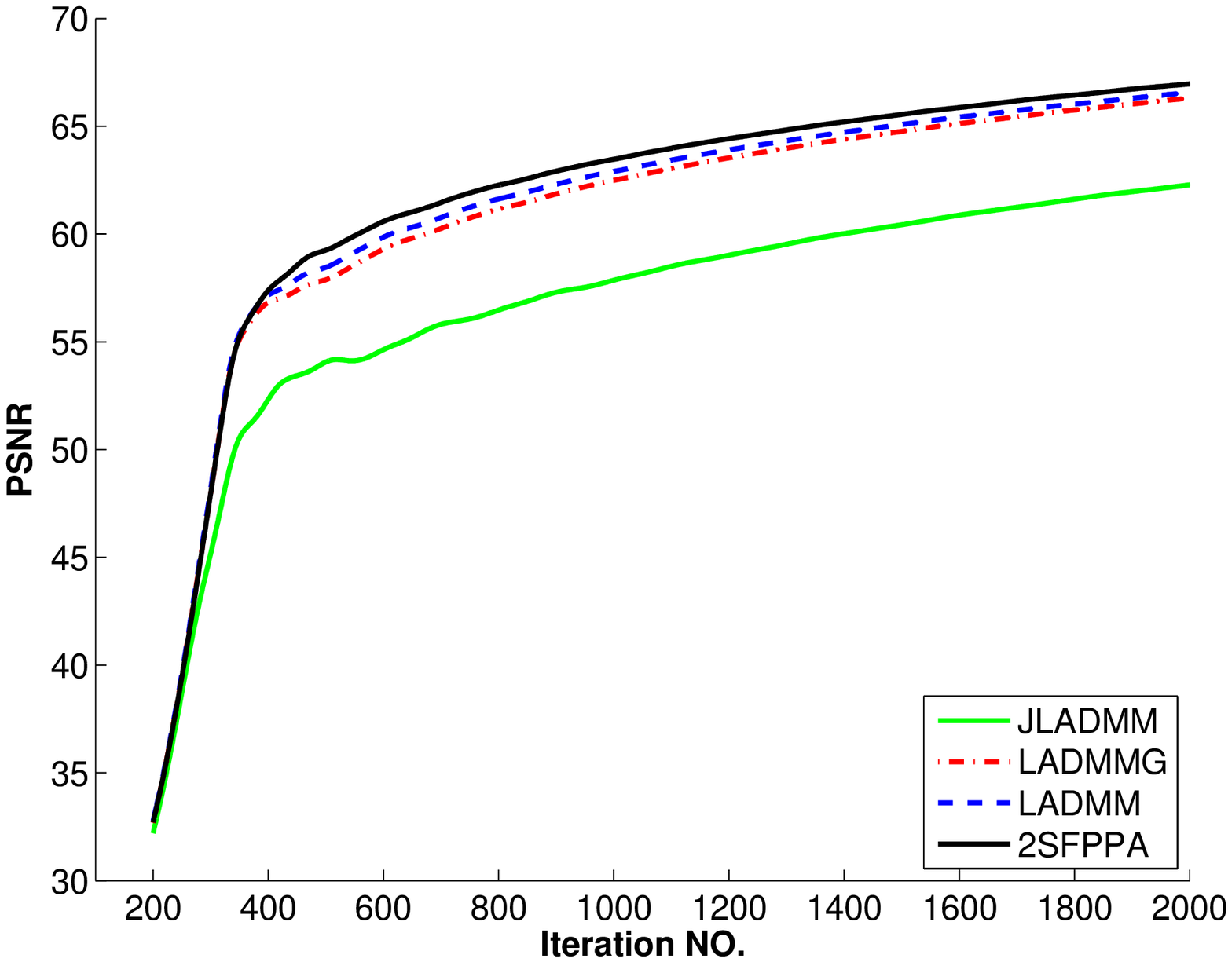}}
&\scalebox{0.4}{\includegraphics{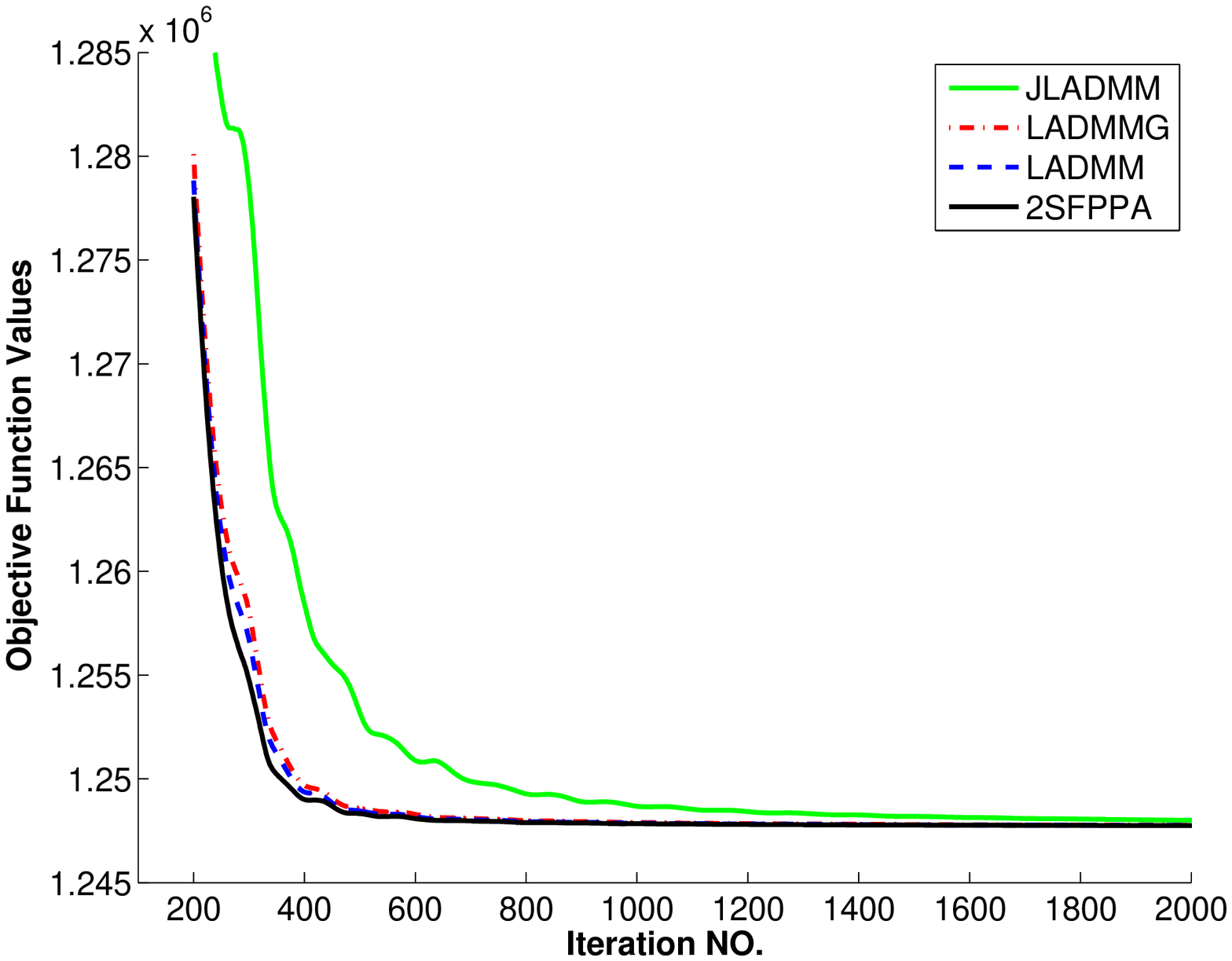}}\\
(c) & (d)\\
\end{tabular}
\caption{\footnotesize{(a) PSNR versus computational time, (b) objective function value versus computational time,  (c)  PSNR versus number of iterations, (d) objective function value versus number of iterations. }}
\end{center}
\end{figure}
\section{Conclusions}\label{sec:conclusion}
In this paper, we study the multi-block separable convex problem, which minimizes the sum of several convex functions with linear constraints. We develop a two-step fixed-point iterative scheme for solving the problem. We prove that the iterative scheme is convergent and has the convergence rate of $O(\frac{1}{k})$ in the ergodic sense and the sense of the partial primal-dual gap, where $k$ denotes the iteration number.  Based on the iterative scheme, we propose a class of convergent two-step algorithms for the multi-block separable convex problem.  Convergence analysis for the specific algorithms can be carried out by verifying conditions on the matrices used to construct the algorithms. In the numerical experiments, we applied our two-step algorithms to the sparse MRI problems. Numerical results show that our proposed algorithms perform as efficiently as LADMM and outperform the JLADMM and LADMMG.

\vspace{10mm}

\bibliographystyle{siam}


\end{document}